\documentclass[final,leqno,onefignum,onetabnum]{siamltex1213}

\usepackage{amsmath,amssymb}
\usepackage{accsupp}
\usepackage{nicefrac}
\usepackage{enumitem}
\usepackage[only,llbracket,rrbracket]{stmaryrd}
\usepackage[ntheorem]{empheq}
\usepackage{algorithm}
\usepackage{algpseudocode}
\usepackage{subcaption}
\usepackage{soul}
\usepackage{epsfig}
\usepackage{./shortcuts}
\usepackage{booktabs}

\title{Symmetric Interior Penalty Discontinuous Galerkin Discretisations
and Block Preconditioning for Heterogeneous Stokes Flow\thanks{Author
S.~M.~Schnepp acknowledges financial support from the Swiss University Conference and the Swiss Council of Federal Institutes of Technology through the 
Platform for Advanced Scientific Computing (PASC) program.
Author D.~E. Charrier acknowledges financial support from the
European Union’s Horizon 2020 research and innovation programme under grant
agreement No 671698.}} 

\author{D.E.~Charrier\footnotemark[2]\ \footnotemark[4]\ \footnotemark[3]
\and D.A.~May\footnotemark[4]\ \footnotemark[5]
\and S.M.~Schnepp\footnotemark[4]}

\begin{document}
\maketitle

\renewcommand{\thefootnote}{\fnsymbol{footnote}}

\footnotetext[2]{Fachbereich Mathematik, Technische Universit{\"a}t Darmstadt, AG Numerik und Wissenschaftliches Rechnen, Dolivostra{\ss}e 15, 64293 Darmstadt, Germany}
\footnotetext[3]{School of Engineering and Computing Sciences, Durham University, Lower Mountjoy South Road, Durham DH1 3LE, United Kingdom (\email{dominic.e.charrier@durham.ac.uk})}
\footnotetext[4]{Institute of Geophysics, ETH Z{\"u}rich, Sonneggstrasse 5, 8092
Z{\"u}rich, Switzerland (\email{mail@saschaschnepp.net})}
\footnotetext[5]{Department of Earth Sciences, University of Oxford, South Parks Road, OX1 3AN
Oxford, United Kingdom (\email{david.may@earth.ox.ac.uk})}

\renewcommand{\thefootnote}{\arabic{footnote}}

\slugger{sisc}{201x}{xx}{x}{x--x}

\definecolor{cdmay}{rgb}{.6,.1,.1}
\definecolor{cdec}{rgb}{.1,.4,.1}
\definecolor{csms}{rgb}{.1,.1,.4}
\newcommand{\act}[1]{\textcolor{cdmay}{#1}}
\newcommand{\actdec}[1]{\textcolor{cdec}{#1}}
\newcommand{\actsms}[1]{\textcolor{csms}{#1}}

\begin{abstract}
Provable stable arbitrary order symmetric interior
penalty discontinuous Galerkin (SIP) discretisations of heterogeneous, 
incompressible Stokes flow utilising $Q^2_k$--$Q_{k-1}$ elements
and hierarchical Legendre basis polynomials are developed and investigated.
For solving the resulting linear system, a block preconditioned iterative
method is proposed. 
The nested viscous problem is solved by a
$hp$-multilevel preconditioned Krylov subspace method.
For the $p$-coarsening, a twolevel method
utilising element-block Jacobi preconditioned iterations as a smoother
is employed. Piecewise bilinear ($Q^2_1$) and piecewise constant
($Q^2_0$) $p$-coarse spaces are considered.
Finally, Galerkin $h$-coarsening is proposed and investigated for
the two $p$-coarse spaces considered.
Through a number of numerical experiments, we demonstrate that utilising the
$Q^2_1$ coarse space results in the most robust $hp$-multigrid method for
heterogeneous Stokes flow.
Using this $Q^2_1$ coarse space we observe that the convergence of the overall
Stokes solver appears to be robust with respect to the jump in the viscosity and only mildly depending on the
polynomial order $k$.
It is demonstrated and supported by theoretical results that the convergence of
the SIP discretisations and the iterative methods
rely on a sharp choice of the penalty parameter based on local values
of the viscosity.
\end{abstract} 



\pagestyle{myheadings}
\thispagestyle{plain}
\markboth{D.E. Charrier et al.}{SIP AND PRECONDITIONING FOR HETEROGENEOUS STOKES}

\section{Introduction}
\subsection{Background and Motivations}
Earth exhibits a diverse range of unique geological processes:
mountain building, subduction and continental rifting, earthquakes and volcanism. These
phenomena are the result of multi-phase, history-dependent, large-deformation processes
spanning million year time scales.

Computational models provide a viable technique to study the evolution in both space and time of geological
processes. 
A prototypical continuum description of the behaviour of rocks is stationary,
incompressible Stokes flow with Boussinesq approximation
\cite{turc.torr.431.1973, schubert_mantle_2001}:
\begin{equation}
\begin{split}
- \textbf{div} \big[ \bar{\eta}(\uvec,p,T,\theta(x)) \, \strain{\uvec}
\big] + \nabla p &= \rho_0(\theta) \big( 1 - \alpha(\theta) ( T - T_0) \big) \hat{\boldsymbol g} \\
- \textup{div}(\uvec) &= 0,
\end{split}
\end{equation}
where $\uvec$, $p$, $\strain{\cdot}$ is the velocity, pressure and strain rate,
respectively,
$T$ is the temperature,
$\theta$ is the material composition,
$\bar{\eta}$ the effective viscosity,
$\rho_0$ is the reference density at reference temperature $T_0$, and
$\hat{\boldsymbol g}$ is the gravity vector.
The conservation of momentum and mass for the creeping fluid is coupled with the
conservation of energy equation:
\begin{equation}
\rho_0(\theta) C_p \frac{DT}{Dt} = \textup{div} \big( k(\theta) \nabla T \big) +
Q(\theta),
\end{equation}
where $C_p$ is the specific heat at constant pressure, $k$ the conductivity, and
$Q$ the external heat source; and the evolution of the composition:
\begin{equation}
\frac{\partial \theta}{\partial t} + \uvec \cdot \nabla \theta = 0.
\end{equation}
From high-pressure and temperature laboratory experiments of minerals, it is known that rocks exhibit
thermally activated creep and follow an Arrhenius type law
\cite{ranalli1995rheology, karato1997phase}:
\begin{equation}
\bar{\eta}(u,p,T,\theta) = A(\theta) \varepsilon_{II}^{n(\theta)} \, \exp\left[ \frac{E(\theta) + pV(\theta)}{n(\theta)RT} \right],
\end{equation}
where $A$ is a compositional dependent experimentally determined constant, $\varepsilon_{II}$ is
the second invariant of the strain rate tensor, $E, V$ are the activation energy and activation volume,
$R$ is the universal gas constant and $n$
is the power-law exponent.
To facilitate brittle behaviour at low temperature, the ductile creep laws are
augmented with a plasticity model (e.g. Drucker-Prager \cite{ranalli1995rheology}).

When such a composite flow law is applied to geodynamics scenarios,
the effective viscosity ($\bar{\eta}$) is highly heterogeneous.
At depths $> 200$ km, ductile behaviour dominates and the viscosity profile can to first order be characterised by a smooth,
exponential function. Above, due to material failure or compositional variations
associated with crustal layers, the viscosity profile will be discontinuous and
can possess jumps on the order of $10^4$--$10^8$ Pa s.
Realistic forward models of both mantle- and crust-scale simulations are adversely affected
by the degree of heterogeneity within the viscosity due to both accuracy concerns associated with the
particular spatial discretisation used, and a lack of solver (linear and nonlinear) robustness.
\subsection{Related work}
Geodynamics forward models of incompressible Stokes which permit highly heterogeneous viscosity structures
have traditionally utilised finite difference (FD) (e.g. \cite{wein.schm.425.1992, Zaleski199255}),
finite volume (e.g. \cite{Gerya03, stem.etal.223.2006,tack.7.2008}),
or finite element (FE) (e.g. \cite{poli.podl:gji.553.1992, full:gji.1.1995, more.dufo.ea:jcp.476.2003, Popov200855, braun2008douar, GJI:GJI5164, may2015scalable})
spatial discretisations.
The relative merits of FD and FE methods for geodynamics applications can be broadly summarised as follows:

Staggered grid FD methods are ``cheap'' (few non-zeros in the stencil),
the general implementation is rather straight-forward.
However, geometric flexibility is limited, and boundary condition imposition
is non-trivial. Introducing new physics may further require the development of a
modified stencil;
see e.g. \cite{gerya2013adaptive, Duretz01032016}. 
In the context of nonlinear problems, Newton linearisation causes stencil
growth and thus increases the overall cost of the discretisation.

Whilst being more expensive than FD methods (on the same grid),
inf-sup stable FE methods permit
geometric versatility and natural boundary conditions are trivial to impose.
Spatial adaptivity ($h$) can be readily introduced without requiring redeveloping the 
underlying numerical method (e.g. \cite{leng_zhon_Q04006_2011,
burs.etal.1691.2009, kronbichler_high_2012}). Newton linearisation does not
cause the equivalent of stencil growth.
For both FD and FE discretisations, robust multi-level preconditioners suitable
for highly heterogeneous viscosity structures exist \cite{tack.7.2008, burs.etal.1691.2009, kronbichler_high_2012, may2015scalable}.

Inf-sup stable discontinuous Galerkin (DG) methods of the interior penalty type
for the Stokes equations can be constructed by using the tensor product element pairs
$Q^2_k$--$Q_{k-1}$, and $Q^2_k$--$Q_{k-2}$ \cite{toselli_hp_2002,schotzau_mixed_2002}, as well as the $H(\text{div};\Omega)$-conforming Raviart-Thomas,
Brezzi-Douglas-Marini, and Brezzi-Douglas-Fortin-Marini kind
element pairs; see
\cite{wang_new_2007,cockburn_note_2007,ayuso_de_dios_simple_2014,kanschat_multigrid_2015}.

The fact that adaptivity in space and approximation order ($k$ or ``$p$'' in the
preconditioner context) is realised with comparably less effort than for other
discretisations makes DG very appealing for geodynamics applications given the very nature
of mantle-lithosphere-crust systems.
The inf-sup constants of stable DG discretisations are further not sensitive to
the element aspect ratio which is highly desirable, for example, within a crustal
scale model with a domain spanning $1000 \times 1000 \times 20$ km where the domain
itself possesses a high aspect ratio, or if anisotropic refinement was employed
\cite[Theorem 9]{schotzau_mixed_2004}.

Regarding the solution of the equation system arising from symmetric
interior penalty DG (SIP) \cite{arnold_sipg_1982,arnold_unified_2002}
based discretisations of incompressible Stokes flow, we note that for
$H(\text{div};\Omega)$-conforming discretisations, efficient preconditioners have been introduced very recently
\cite{ayuso_de_dios_simple_2014} \cite{kanschat_multigrid_2015}.
Recent advances in developing efficient and robust solvers for interior penalty
DG discretisations of second order elliptic problems with heterogeneous 
coefficients involve the algebraic multigrid preconditioner proposed in
\cite{bastian_algebraic_2012,bastian_fully-coupled_2014},
as well as the twolevel methods proposed in
\cite{dobrev_twolevel_2006} and
\cite{van_slingerland_fast_2014,van_slingerland_scalable_2015}.

To the best of our knowledge, employing DG
methods for heterogeneous Stokes flow problems in geodynamics was so
far only considered in \cite{lehmann_comparison_2015}.
Preconditioning was not discussed there.
\subsection{Contributions}
We examine the applicability of using mixed SIP based Stokes discretisations
  for studying heterogeneous, incompressible Stokes flow
  problems associated with prototype problems arising in geodynamics.
  Through comparison with an analytic solution employing a discontinuous
  viscosity structure, we numerically demonstrate that the discretisation
  yields optimal order of accuracy for $Q^2_k$--$Q_{k-1}$ elements
  for $k = 1, \dots, 6$.
  
Our main contribution is the development of a preconditioned
  iterative method for the discrete saddle point system resulting from the Stokes
  discretisation. To this end, we follow a block preconditioning approach; e.g.
  \cite{elman_finite_2014}.
  The nested viscous problem is solved by a
  $hp$-multilevel preconditioned Krylov subspace method.
  For the $p$-coarsening, a twolevel method
  utilising element-block Jacobi preconditioned iterations as a smoother
  is employed. Two $p$-coarse spaces are considered:
  the space of element-wise constants and the space
  of continuous, element-wise bilinear functions
  Through numerical experiments with
  heterogeneous viscosity, we demonstrate that the 
  variant utilising the element-wise
  bilinear coarse space has a convergence rate which is 
  independent of the number of elements, 
  largely insensitive to the jump in viscosity, and only weakly dependent on the
  polynomial order.

The heterogeneous nature of the viscosity in geodynamics applications
  requires a careful choice of the SIP penalty parameters.
  We provide a brief analysis of the influence of the
  penalty parameters on the discretisation error as well as on the quality of
  the element-block Jacobi smoothers.
\subsection{Limitations}
We restrict ourselves to linear problems with element-wise constant viscosity
distributions in this study.
   
\section{Governing equations} \label{sec:governing_equations}
Neglecting nonlinearities and the effect of temperature on the
viscosity, we restrict ourselves to a model
that is solely depending on the material composition of
the rocks in the Earth's mantle:
\begin{subequations}
\begin{align}
\label{eq:governing_equations:momentum_balance}
-\textbf{div}(2\,\eta(\theta)\,\strain{\uvec}) + \nabla\pscalar &=
\vec{f}(\theta) &&\inDomain,
\\
\label{eq:governing_equations:incompressibility_constraint}
\text{div}(\uvec) &= 0
&&\inDomain,
\end{align}
where $\uvec$ is a velocity field, $p$ is a
pressure, $\eta$ denotes the viscosity, $\vec{f}$ denotes a volumetric force,
$\theta$ denotes the
material composition with $0\leq\theta\leq1$, and $\strain{\uvec}$ denotes the
(linearised) strain rate tensor with
$\strainT{\uvec}{ij}=\frac{1}{2}(\pddx{j}{\uvecel{i}}+\pddx{i}{\uvecel{j}})$,
$i,j=1,\ldots,d$.
Further, $\domain \subset \mathbb{R}^2$, denotes a rectangular domain
with boundary $\boundary=\overline{\boundaryM\cup\boundaryN}$
consisting of a Neumann part $\boundaryN$, and a Navier part $\boundaryM$.
We require the velocity and pressure to satisfy homogeneous Neumann boundary
conditions,
\begin{align}
\label{eq:governing_equations:neumann_bc}
(2\,\eta\,\strain{\uvec} - \pscalar\identity)\normal &= 0
&&\onBoundaryN,		
\intertext{as well as homogeneous Navier
boundary conditions,}
\label{eq:governing_equations:navier_dirichlet_bc}
\uvec\cdot\normal &= 0,\qquad\tangential\cdot 2\,\eta\,\strain{\uvec}\normal
= 0 && \onBoundaryM,
\end{align} 
where $\normal$ denotes the outward normal to the boundary, and
$\tangential$ denotes a vector belonging to the tangential space of
$\boundary$. 
In case the Neumann boundary is empty, we ensure uniqueness of the pressure solution by enforcing the following constraint:
\begin{align}
\label{eq:governing_equations:pressure}
\intDomain p\,\dV=0.
\end{align}
\end{subequations}
We further require that the domain $\domain$ is restrained against rigid
motions
\namedlabel{itm:governing_equations:assumption_rigid_motions}{(A1)}. 

Let us denote by $\lebesgueTwo{\domain}$ and $\sobolev{1}{\domain}$
the usual Sobolev spaces and by
$\normLebesgueTwo{\cdot}{\domain}$
and
$\normSobolev{\cdot}{1}{\domain}$
their norms.

Let us assume that the viscosity
$\eta\in\lebesgueTwo{\domain}$ is bounded according to
\begin{align}
\label{eq:stokes:assumption_eta_bounds}
0 < \etaMin \leq \eta \leq \etaMax,
\end{align}
\namedlabel{itm:stokes:assumption_eta_bounds}{(AS1)},
and that $\vec{f}\in\lebesgueTwo{\domain}^2$
\namedlabel{itm:stokes:assumption_force}{(AS2)},
then it is well-known that problem
\eqref{eq:governing_equations:momentum_balance} --
\eqref{eq:governing_equations:navier_dirichlet_bc}
s.t. to the other named constraints admits a unique weak solution; see 
e.g.~\cite{brezzi_mixed_1991}. 
\section{Computational grid and trace operators}
\label{sec:grid}
Let $\tria$ be a regular 
Cartesian grid on $\domain$
\namedlabel{itm:grid:assumption_grid}{(Ah1)}.
We refer to the disjoint open sets $\cell \in \tria$ as elements and denote
their diameter by $h$. The number of elements is denoted by $N_\cell$.
Finally, $\normal$ denotes the outward normal unit vector to the element boundary $\cellBnd$.

An interior face of $\tria$ is the $d-1$ dimensional intersection
$\cellBnd^{+} \cap \cellBnd^{-}$, where $\cellPlus$ and $\cellMinus$
are two adjacent elements of $\tria$. Similarly, a boundary face of $\tria$ is
the $d-1$ dimensional intersection $\cellBnd \cap \boundary$ which consists
of entire faces of $\cellBnd$.
We denote by $\internalFaces$ the union of all interior faces of $\tria$ ,
by $\skeletonN$, and $\skeletonM$ the union of all boundary
faces belonging to the Neumann part, and the Navier part
of the boundary, respectively, and set $\skeleton = \internalFaces \cup \skeletonN \cup
\skeletonM$.
Here and in the following, we refer generically to a ``face" although we
consider only two-dimensional problems in this paper.
%
%

Let $\cell \in \tria$; we denote by $\sobolev{s}{K}$ the space of
real-valued functions $\scal{v} \in \lebesgueTwo{K}$ such that the function $\scal{v}$ and its
weak derivatives up to order $s$ are measurable
and square integrable in $K$.
We will denote the norms on all three spaces $\sobolev{s}{K}$,
$\sobolev{s}{K}^2$, and $\sobolev{s}{K}^{2 \times 2}$ by the symbol
$\norm{\cdot}_{s,K}$. We further
introduce the
broken Sobolev space
\begin{align}
\sobolev{s}{\tria} &= \{ \scal{v} \in \lebesgueTwo{\domain} \colon\,
\scal{v}|_{\cell} \in \sobolev{s}{\cell},\,\cell \in \tria\}.
\end{align}
Let us denote the norms of $\sobolev{s}{\tria}$, $\sobolev{s}{\tria}^2$,
and $\sobolev{s}{\tria}^{2 \times 2}$
by the symbol
$\norm{\cdot}_{\sobolev{s}{\tria}}$.

Let $q\in\sobolev{1}{\tria}$, and $\varphi$ either belong to
$\sobolev{1}{\tria}$, $\sobolev{1}{\tria}^2$, or
$\sobolev{1}{\tria}^{2 \times 2}$.
Let $\face \in \internalFaces$ be an interior face shared by
the elements $\cellPlus$ and $\cellMinus$.
Let $\varphi^{\pm}$ and $q^{\pm}$ denote the traces of $\varphi$ and $q$ on
$\face$ from the interior of $\cell^{\pm}$, respectively. 
Further, let $\normal^{\pm}$ denote the outward
normal unit vector to the boundary $\cellBnd^{\pm}$.
We define the mean value $\avg{\varphi}$ 
and the jump $\jmp{q\normal}$ at $\vec{x} \in \face$ by
\begin{align}
\avg{\varphi} &= \frac{1}{2} (\varphi^{+} + \varphi^{-}),\quad
\jmp{q \normal} = q^{+} \normal^{+} +
q^{-} \normal^{-}.
\end{align}
Let $\vec{w}$ denote a vector-valued function in
$\sobolev{1}{\tria}^2$,
and $\vec{w}^{\pm}$ denote its traces on $\face$ from the
interior of $\cell^{\pm}$. We define the jumps $\jmp{\vec{w}
\otimes \normal}$ and $\jmp{\vec{w}
\cdot \normal}$ at $\vec{x} \in \face$ by
\begin{align}
\jmp{\vec{w}\otimes \normal} &= \vec{w}^{+}\otimes\normal^{+} +
\vec{w}^{-}\otimes\normal^{-},&
\jmp{\vec{w}\cdot \normal} &= \vec{w}^{+}\cdot\normal^{+} +
\vec{w}^{-}\cdot\normal^{-},
\end{align}
where ``$\otimes$'' denotes the dyadic product.
\section{Discretisation of the Stokes problem}
\label{sec:discretisation_stokes}
Let us introduce $\etaCell{\cell}=\eta|_\cell$,
$\cell\in\tria$, and define:
\begin{align}
\label{eq:stokes_discretisation:eta_face}
\etaFace &= \begin{cases}
\max \left\{ \etaMaxCell{\cellPlus}, \etaMaxCell{\cellMinus}\right\} &
\face\in\internalFaces,\\
\etaMaxCell{\cell} & \face\in\skeletonM.
\end{cases}
\end{align}
For simplicity, we assume here that the viscosity is element-wise
constant \namedlabel{itm:stokes_discretisation:assumption_eta}{(ASh1)}.
Let us additionally define the constants
\begin{align}
\label{eq:stokes_discretisation:const_trace_face}
\constTraceFace &=
\begin{cases}
\max\left\{\constTraceCell{\cellPlus},\constTraceCell{\cellMinus}\right\} &
\face\in\internalFaces, \\
\constTraceCell{\cell} & \face\in\skeletonM.
\end{cases}
\end{align}
where the constants $\constTraceCell{\cell}$, $\cell\in\tria$, stem 
from the following discrete trace inequality \cite{hillewaert_development_2013}:
\begin{lemma}[Discrete trace inequality]
\label{thm:stokes_discretisation:discrete_trace}
Let $K$ be an affine quadrilateral,
and let $e$ be an edge belonging to the boundary of $K$.
Then, it holds that
\begin{align}
\label{eq:stokes_discretisation:discrete_trace}
\norm{\varphi_h}_{\lebesgueTwo{e}}^2 &\leq
\constTraceCell{\cell}\,
\norm{\varphi_h}_{\lebesgueTwo{K}}^2,\qquad\forall\varphi_h\in Q_k(K),
\intertext{with the trace inequality constant}
\label{eq:stokes_discretisation:discrete_trace_constant}
\constTraceCell{\cell} &=
(k+1)^2
\,\frac{|e|}{|K|}.
\end{align}
\end{lemma}

We approximate velocity and pressure in the discontinuous finite
element spaces
\begin{align}
\spaceVhf &= \left\{ \vvec \in \lebesgueTwo{\domain}^2\colon\,\,
\vvec|_{\cell} \in Q_{\kcell}(\cell)^2,\,\cell \in \tria \right\},
\\
\spaceQhf &= 
\begin{cases}
\left\{ \qscalar \in \lebesgueTwo{\domain}\colon\,\,\qscalar|_{\cell}
\in Q_{\kcell-1}(\cell),\,\cell \in \tria \right\} & |\boundaryN| > 0,  \\
\left\{ \qscalar \in \lebesgueTwo{\domain}\setminus
\mathbb{R}\colon\,\,\qscalar|_{\cell} \in Q_{\kcell-1}(\cell),\,\cell \in \tria
\right\} & \text{else},  \\
\end{cases}
\end{align}
where $Q_k(K)$ is the space of polynomials of maximum degree $k$ in each
variable on the mesh cell $K\in\tria$.

As approximation to \eqref{eq:governing_equations:momentum_balance} --
\eqref{eq:governing_equations:pressure},
we then consider the problem of finding $\vec{\uhvec}\in\spaceVhf$ and
$\phscalar\in\spaceQhf$ such that:
\begin{align}
\label{eq:stokes_discretisation:discrete_momentum_balance}
\dgformA{\uhvec}{\vhvec} + \dgformB{\vhvec}{\phscalar} &= \dgformF{\vhvec}, &&
\forall \vhvec \in \spaceVhf,\\
\label{eq:stokes_discretisation:discrete_incompressible_constraint}
\dgformB{\uhvec}{\qhscalar} \phantom{\,\,\,\,\,\,+ \dgformB{\vhvec}{\pscalar}}
&= \dgformG{\qhscalar}, && \forall \qhscalar \in \spaceQhf,
\end{align}
where we use a SIP form $\mathcal{A}_h$, and a form
$\mathcal{B}_h$ similar to the one used in \cite{toselli_hp_2002}:
\begin{align}
\label{eq:stokes_discretisation:stiffness_form}
\dgformA{\uhvec}{\vhvec}  =
\sumCells
\intCell 2\,\eta\,\strain{\uhvec} : \strain{\vhvec}\,\dV
\notag
\\
-
\sumInternalFaces \intFace \avg{2\,\eta\,\strain{\uhvec}} : \jmp{\vhvec \otimes
\normal}\,\ds
&-\sumInternalFaces\intFace \avg{2\,\eta \,\strain{\vhvec}} :
\jmp{\uhvec \otimes \normal}
\,\ds
\\
- \sumSkeletonM\intFace (\normal\cdot 2\,\eta \,\strain{\uhvec}\normal)\,
(\vhvec \cdot \normal)\,\ds
&-\sumSkeletonM\intFace (\normal\cdot 2\,\eta
\,\strain{\vhvec}\normal)\, (\uhvec \cdot \normal)
\,\ds
\notag\\
+ \sumInternalFaces \penalt\intFace
\jmp{\uhvec\otimes\normal}:\jmp{\vhvec\otimes\normal}\,\ds
&+ \sumSkeletonM \penalt\intFace
(\uhvec\cdot\normal)\,(\vhvec\cdot\normal)\,\ds,\notag
\end{align}
and
\begin{align}
\label{eq:stokes_discretisation:divergence_form}
\dgformB{\uhvec}{\qhscalar} &=
\sumCells \intCell
-\text{div}(\uhvec)\,\qhscalar\,\dV
+ \sumInternalFaces \intFace \avg{\qhscalar}\,\jmp{\uhvec\cdot \normal}\,\ds\\
&+ \sumSkeletonM \intFace \qhscalar\,(\uhvec\cdot \normal)\,\ds
\notag,
\\
\label{eq:stokes_discretisation:rhs_f}
\dgformF{\vhvec} &= \sumCells\intCell \vec{f}\cdot\vhvec\,\dV
,
\qquad
\dgformG{\qhscalar} = 0,
\end{align}
with $\uhvec,\vhvec\in\spaceVhf$, $\qhscalar\in\spaceQhf$,
and the face-wise penalties $\penalt=\penal\,\constTraceFace$,
$\face\in\skeleton$.
The parameters $\penal$ are the so-called penalty or stability parameters that
must be chosen sufficiently large (to be specified below) to guarantee that the
bilinear form $\dgformAs$ is coercive on the discrete space $\spaceVhf$.

Consistency of the discrete variational problem can be shown by following
the proof of \cite[Lemma 7.5.]{toselli_hp_2002}.
\subsection{Stability of the Stokes discretisation}
\label{sec:stokes:discrete_stability}
For the analysis of the Stokes discretisation, it is necessary to introduce the
functionals
\begin{align}
\label{eq:grid:norm_0}
\normZero{\qscalar}^2 &=  \sumCells\normLebesgueTwo{\qscalar}{\cell}^2 +
\sumInternalFaces \kface^{-2}\,h\,\normLebesgueTwo{\avg{\qscalar}}{\face}^2
+
\sumSkeletonM \kface^{-2}\,h\,\normLebesgueTwo{{\qscalar}}{\face}^2,
\\
\label{eq:grid:norm_h}
\normh{\vvec}^2 &=
\sumCells\norm{\nabla\vvec}_{\lebesgueTwo{\cell}}^2
+ \sumInternalFaces
\kface^2\,h^{-1}\,\norm{\jmp{\vvec\otimes\normal}}_{\lebesgueTwo{\face}}^2,
\\
&+ \sumSkeletonM
\kface^2\,h^{-1}\,\norm{\vvec\cdot\normal}_{\lebesgueTwo{\face}}^2,
\notag
\\
\label{eq:grid:norm_hh}
\normhh{\vvec}^2 &=
\normh{\vvec}^2
+
\sumInternalFaces \kface^{-2}\,h\,
\normLebesgueTwo{\avg{\nabla\,\vvec}}{\face}^2
\sumSkeletonM \kface^{-2}\,h\,
\normLebesgueTwo{\nabla\,\vvec}{\face}^2,
\end{align}
with $\vvec\in\sobolev{2}{\tria}^2$ and
$\qscalar\in\lebesgueTwo{\domain}$.
By definition, $\normZero{\cdot}$ is a norm
on $\lebesgueTwo{\domain}$.
Under assumption \ref{itm:governing_equations:assumption_rigid_motions},
we assume that $\normh{\cdot}$ is a norm
on $\sobolev{2}{\tria}^2$ for the considered boundary conditions.
\namedlabel{itm:stokes_discretisation:assumption_norm_h}{(ASh2)}.

Stability of the Stokes discretisation follows from the
discrete inf-sup stability of $\mathcal{B}_h$, the discrete coercivity of
$\mathcal{A}_h$, and 
Brezzi's lemma:
\begin{lemma}[Discrete inf-sup
stability\label{thm:stokes_discretisation:discrete_inf_sup}]
Let the assumptions \ref{itm:grid:assumption_grid}
be satisfied and let $\kcell \geq 1$, $\cell\in\tria$. Then, it holds that
\begin{align}
\sup_{\vec{0} \neq \uhvec\in \spaceVhf}
\frac{\dgformB{\uhvec}{\qhscalar}}{\normh{\uhvec}} &\geq
\beta_h\,\norm{\qhscalar}_{\lebesgueTwo{\domain}}
=c\,k^{-1}\,\norm{\qhscalar}_{\lebesgueTwo{\domain}},&&
\forall \qhscalar \in\spaceQhf,
\end{align}
with $c > 0$ independent of $h$ and $k$.
\end{lemma}
\begin{proof} This follows from \cite[Theorem
6.2]{schotzau_mixed_2002} for $k\geq 2$. The case $k=1$ is covered by
\cite[Theorem 6.12.]{schotzau_mixed_2002}.
\end{proof}

We rely on a discrete Korn inequality to
show coercivity of the bilinear form $\dgformAs$.
\begin{lemma}[Discrete Korn
inequality,\cite{brenner_korns_2004}\label{thm:stokes_discretisation:discrete_korn}]
Let the assumptions \ref{itm:governing_equations:assumption_rigid_motions} and
\ref{itm:grid:assumption_grid}
be satisfied. Then, it holds for all $\vvec \in \sobolev{1}{\tria}^2$ that
\begin{align}
\label{eq:stokes_discretisation:discrete_korn}
\left(
\norm{\strain{\vvec}}_{\lebesgueTwo{\domain}}^2
+
\sum_{\face \in \internalFaces}
h^{-1}\,\norm{\jmp{\vvec \otimes \normal}}_{\lebesgueTwo{\face}}^2
\right)
&\geq
\constKornDisc\,\norm{\nabla\vvec}_{\lebesgueTwo{\domain}}^2
\end{align}
with a constant $\constKornDisc>0$ independent of $h$.
\end{lemma}
\begin{proof}
This is a slight modification to the result in \cite{brenner_korns_2004} relying
on assumption \ref{itm:governing_equations:assumption_rigid_motions} as well
as on the affinity and shape regularity of the grid elements
\ref{itm:grid:assumption_grid}.
\end{proof}
\begin{lemma}[Discrete coercivity\label{thm:stokes_discretisation:discrete_coercivity}] Let the assumptions
\ref{itm:governing_equations:assumption_rigid_motions},
\ref{itm:grid:assumption_grid},
and
\ref{itm:stokes_discretisation:assumption_eta} --
\ref{itm:stokes_discretisation:assumption_norm_h} hold.
Assume that the penalty parameters are chosen according to
\begin{align}
\label{eq:stokes_discretisation:sipg_penalty}
\penal >
\nCellFaces\,\etaFace\,
(1+\stabilisationKorn)
\,
\begin{cases}
1\,&\face\in\internalFaces,\\
2\,&\face\in\skeletonM,
\end{cases} 
\end{align}
with $\etaFace$ from \eqref{eq:stokes_discretisation:eta_face},
and with $\nCellFaces=4$ denoting the number of faces of an element.
Then, it holds that 
\begin{align}
\dgformA{\vhvec}{\vhvec} &\geq \dgalpha\,\normh{\vhvec}^2,
\qquad\forall\vhvec\in\spaceVhf,
\end{align}
with a constant $\dgalpha = C\,2\,\etaMin >
0$ independent of $h$ and $k$. Here,
$\etaMin$ is from assumption
\ref{itm:stokes:assumption_eta_bounds}.
The parameter $\tau>0$ is a small value that is necessary
to use the discrete Korn inequality.
\end{lemma}
\begin{proof}
The proof is based on the most part on the stability analysis
for the \textit{element-wise} penalty approach presented in \cite[Section
3.3.3.]{drosson_stability_2013}.
Let $\uhvec=\vhvec$, $\uhvec,\vhvec\in\spaceVhf$; we apply a Young's inequality
to the second and third term of \eqref{eq:stokes_discretisation:stiffness_form}.
Thus, we obtain
\begin{align}
\label{eq:stokes_discretisation:coercivity_proof_1}
\dgformA{\vhvec}{\vhvec} \geq
\sumCells\intCell 2\,\eta\,\strain{\vhvec}: \strain{\vhvec}&\,\ds
-\sumInternalFaces\epsilon_\face^{-1}\,
\int_\face\,\avg{2\,\eta\,\strain{\vhvec}}^2\,\ds
\notag\\
-\sumSkeletonM\epsilon_\face^{-1}\,
\int_\face\,(2\,\eta\,\strain{\vhvec})^2\,\ds
&+\sumInternalFaces(\penalt-\epsilon_\face)\,\normLebesgueTwo{\jmp{\vhvec\otimes\normal}}{\face}^2
\notag \\
&+\sumSkeletonM(\penalt-\epsilon_\face)\,\normLebesgueTwo{\vhvec\cdot\normal}{\face}^2,
\end{align}
where $\epsilon_\face > 0$, $\forall
\face\in\internalFaces\cup\skeletonM$.
As a next step, we will bound the second and third term of
\eqref{eq:stokes_discretisation:coercivity_proof_1} below. We have for interior faces
$\face\in\internalFaces$,
\begin{align}
\label{eq:stokes_discretisation:coercivity_proof_2}
\intFace\,\avg{2\,\eta\,\strain{\vhvec}}^2\,\ds
&\leq
\int_\face\,\frac{1}{2}
\left((2\,\eta\,\strain{\vhvec})|_\cellPlus^2
+(2\,\eta\,\strain{\vhvec})|_\cellMinus^2
\right)\,\ds
\\
&\overset{\text{(I)}}\leq
2\,\etaFace\,
\frac{1}{2}
\constTraceFace\,
\Bigg(
\int_\cellPlus\,
2\,\eta\,\strain{\vhvec} : \strain{\vhvec} \,\dV
\notag
\\
&+
\int_\cellMinus\,
2\,\eta\,\strain{\vhvec} : \strain{\vhvec} \,\dV
\Bigg)
\notag
.
\end{align}
In step (I), we have applied the trace inequality
\eqref{eq:stokes_discretisation:discrete_trace} from lemma
\eqref{thm:stokes_discretisation:discrete_trace}, and further have introduced
constant $\constTraceFace$ defined as in
\eqref{eq:stokes_discretisation:const_trace_face},
as well as constant $\etaFace$
defined as in \eqref{eq:stokes_discretisation:eta_face}.
Analogously, we obtain for boundary
faces
$\face\in\skeletonM$ ,
\begin{align}
\label{eq:stokes_discretisation:coercivity_proof_3b}
\intFace\,(2\,\eta\,\strain{\vhvec})^2\,\dV
&\leq
2\,\etaFace\,
\constTraceFace\,
\int_\cell\,
2\,\eta\,
\strain{\vhvec}:\strain{\vhvec}
\,\ds.
\end{align}
Inserting
\eqref{eq:stokes_discretisation:coercivity_proof_2}
--
\eqref{eq:stokes_discretisation:coercivity_proof_3b}
in \eqref{eq:stokes_discretisation:coercivity_proof_1}, leads to
\begin{align}
\label{eq:stokes_discretisation:coercivity_proof_4}
\dgformA{\vhvec}{\vhvec} &\geq
\sumCells\intCell 2\,\eta\,\strain{\vhvec} : \strain{\vhvec}\\
\times
\Bigg(
1 - \sum_{\face\subset\cellBnd,\face\in\internalFaces}
\epsilon_\face^{-1}\,
\etaFace\,
&
\constTraceFace
\notag
-
\sum_{\face\subset\cellBnd,\face\in\skeletonM}
\epsilon_\face^{-1}\,
2\,\etaFace\,
\constTraceFace\,
\Bigg)\,\dV
\notag\\
&+\sumInternalFaces(\penal\,\constTraceFace-\epsilon_\face)\,\normLebesgueTwo{\jmp{\vhvec\otimes\normal}}{\face}^2
\notag\\
&+\sumSkeletonM(\penal\,\constTraceFace-\epsilon_\face)\,\normLebesgueTwo{\vhvec\cdot\normal}{\face}^2,
\; \vhvec\in\spaceVhf.
\notag
\end{align}
We see that this expression is positive for any $\vhvec \in
\spaceVhf$ if
\begin{align}
\penal\,\constTraceFace
&> \epsilon_\face >
\penal^{*}\,\constTraceFace =
\nCellFaces\,\etaFace\;
\constTraceFace,
\face\in\internalFaces,
\\
\penal\,\constTraceFace
&> \epsilon_\face
>
\penal^{**}\,\constTraceFace =
\nCellFaces\,2\,\etaFace\,
\constTraceFace,
\;\face\in\skeletonM,
\end{align}
where $\nCellFaces$ denotes the number of faces of
a quadrilateral element.
In the following, we choose the penalty parameters according to
\eqref{eq:stokes_discretisation:sipg_penalty}.
We add a small value $\tau\,\etaFace/\nCellFaces\geq\tau\,\etaMin/\nCellFaces$
once/twice to each penalty parameter $\penal$ in order to use the discrete Korn inequality
\eqref{eq:stokes_discretisation:discrete_korn} from lemma
\ref{thm:stokes_discretisation:discrete_korn} in step (I) of the next derivations. Thus, if we choose the penalty value according to
\eqref{eq:stokes_discretisation:sipg_penalty},
we obtain from \eqref{eq:stokes_discretisation:coercivity_proof_4} that
\begin{align}
\dgformA{\vhvec}{\vhvec} \overset{\text{(I)}}\geq&
C\,
\Bigg(
\constKornDisc\,2\,\etaMin
\sumCells\normLebesgueTwo{\nabla\vhvec}{\cell}^2
\\
&+\sumInternalFaces
2\,\etaFace\,
\constTraceFace\,\normLebesgueTwo{\jmp{\vhvec\otimes\normal}}{\face}^2
\notag\\
&+\sumSkeletonM
2\,\etaFace\,
\constTraceFace\,\normLebesgueTwo{\vhvec\cdot\normal}{\face}^2
\Bigg)\notag
\\
\overset{\text{(II)}}\geq&
C\,
2\,\etaMin\,
\norm{\vhvec}_{1,h}^2,\;\vhvec\in\spaceVhf,\notag
\end{align}where $\etaMin$ is from assumption
\ref{itm:stokes:assumption_eta_bounds}, and
$\constKornDisc$ is from the  discrete Korn inequality
\eqref{eq:stokes_discretisation:discrete_korn} from lemma
\ref{thm:stokes_discretisation:discrete_korn}.
In step (II), we have used that due to the affinity and shape regularity of the grid elements
\ref{itm:grid:assumption_grid}, it holds that
$h^{-1}\leq|\face|/|\cell|\leq\gamma^2\,h^{-1}$
for $\cell\in\tria$ and $\face\subset\cellBnd$, with
$\gamma$ denoting the shape regularity constant.
We have further bounded $\penal$ below by $\etaMin\,\nCellFaces$. The constant
$C$ is assigned a different value in every step and independent of $h$ and $k$.
\end{proof}

We typically choose the penalty parameters as the lower bound 
since estimate \eqref{eq:stokes_discretisation:sipg_penalty} is not totally
sharp due to the utilised inequalities.
Even smaller values can be chosen in practice 
\cite{hillewaert_development_2013}.
We note that the parameter $\tau$ is set to zero in our computations.
We further remark that estimate \eqref{eq:stokes_discretisation:sipg_penalty}
assumes piecewise constant viscosity distributions. For piecewise polynomial viscosity
distributions, it is necessary to replace $k$ by ${k\gets k+m_\eta}$
in the discrete trace inequality constant.
\subsection{A-priori error estimates for the Stokes discretisation}
\label{sec:stokes:discrete_errors}
Let us state the continuity properties of
the forms $\dgformAs$ and $\dgformBs$:
\begin{lemma}[Continuity of $\dgformAs$ and $\dgformBs$\label{thm:stokes_discretisation_error:continuity}]
Under assumptions
\ref{itm:governing_equations:assumption_rigid_motions} --
\ref{itm:stokes:assumption_force},
\ref{itm:grid:assumption_grid},
and
\ref{itm:stokes_discretisation:assumption_eta},
bilinear forms
$\dgformAs$ and $\dgformBs$ are continuous in the sense that
\begin{align}
\label{eq:stokes_discretisation_error:continuity_ah}
\dgformA{\uvec}{\vvec}
&\leq
\Cah\,\normhh{\uvec}\,\normhh{\vvec},
&&\forall(\uvec,\vvec)\in\sobolev{2}{\tria}^2 \times
\sobolev{2}{\tria}^2\\
\label{eq:stokes_discretisation_error:continuity_bh}
\dgformB{\uvec}{\qscalar}
&\leq
\Cbh\,\normh{\uvec}\,\normZero{\qscalar},
&&\forall(\uvec,\qscalar)\in\sobolev{2}{\tria}^2\times\sobolev{1}{\tria}
\end{align}
with $\Cah = \left(2\,\etaMax+\max_{\face\in\internalFaces\cup\skeletonM}\{\penal\}\right) > 0$
and $\Cbh = \sqrt{d} > 0$.
Here, $\etaMax$ is from assumption \ref{itm:stokes:assumption_eta_bounds},
and $\penal$ denotes the penalty value on face $\face$.
Both constants $\Cah$ and $\Cbh$ are independent of $h$ and $k$.
\end{lemma}
\begin{proof}
The proof of both inequalities follows from standard inequalities.
See e.g. the proofs of \cite[Lemma 7.1]{toselli_hp_2002} and
\cite[Lemma 7.2]{toselli_hp_2002} for more details.
\end{proof}

Further note that the norms
$\normh{\cdot}$ and $\normhh{\cdot}$ are equivalent
on the discrete spaces $\spaceVhf$ and $\spaceQhf$. This is also the case for
the norms $\normZero{\cdot}$ and $\normLebesgueTwo{{\cdot}}{\domain}$.

Using the consistency of the discretisation, the discrete inf-sup stability of
$\dgformBs$, the discrete coercivity of $\dgformAs$, the continuity of both
bilinear forms together with the discrete equivalency of norms, as well as
suitable element-wise $hp$-interpolants, we can derive the following a-priori
error estimate:
\begin{lemma} \label{thm:stokes_discretisation_error:error_norm_h}
Let the assumptions.
\ref{itm:governing_equations:assumption_rigid_motions},
\ref{itm:stokes:assumption_eta_bounds}
--
\ref{itm:stokes:assumption_force},
\ref{itm:grid:assumption_grid},
and
\ref{itm:stokes_discretisation:assumption_eta} hold.
Assume that the weak solution
$(\uvec,\pscalar)$
to
\eqref{eq:governing_equations:momentum_balance} --
\eqref{eq:governing_equations:pressure}
belongs to $\sobolev{m}{\cell}^2\times\sobolev{n}{\cell}$,
$\cell\in\tria$,
with $m\geq 2$ and $n\geq1 $.
Further, let $\uhvec\in\spaceVhf$ and $\phscalar\in\spaceQhf$
denote the discrete solution
to problem \eqref{eq:stokes_discretisation:discrete_momentum_balance} --
\eqref{eq:stokes_discretisation:discrete_incompressible_constraint}.
Then, it holds hat
\begin{align}
\normh{\uvec-\uhvec}
+
\normZero{\pscalar-\phscalar}
&\leq
\notag \\
&C\,
\sumCells
\left(
\frac{\Cah}{\dgalpha}\,
\frac{1}{\dgbeta}
\frac{h^{s-1}}{\kcell^{m-3/2}}\normSobolev{\uvec}{m}{\cell}
+
\frac{h^{r}}{\kcell^{n}}\normSobolev{\pscalar}{n}{\cell}
\right),
\end{align}
with $1\leq s \leq \min\{\kcell+1,m\}$ and
$1\leq r \leq \min\{\kcell,n\}$. The constant $C>0$ is independent of
$h$ and $k$ but depends on the shape regularity of the grid elements.
The constant $\Cah$ is from lemma \ref{thm:stokes_discretisation_error:continuity}, and
$\dgalpha$ is from lemma \ref{thm:stokes_discretisation:discrete_coercivity}.
Note that this estimate holds up to the discrete inf-sup constant $\dgbeta$ from
lemma \ref{thm:stokes_discretisation:discrete_inf_sup} that depends on $k$.
\end{lemma}
\begin{proof}
The proof requires arguments analogous to the ones used for the proofs of
\cite[Lemma 8.1]{toselli_hp_2002} and \cite[Lemma 8.2]{toselli_hp_2002}.
\end{proof}

Utilising a duality argument, additional requirements on the regularity
of the solution to the continuous variational problem (and on the solution to
the adjoint problem), as well as $hp$-interpolants for the whole domain
$\domain$, we can derive the following $L^2$ error estimate for the discrete velocity
solution:
\begin{lemma}[$L^2$ error]
\label{thm:stokes_discretisation_error:error_norm_l2_uniform}
Let the assumptions
\ref{itm:governing_equations:assumption_rigid_motions} --
\ref{itm:stokes:assumption_force},
\ref{itm:grid:assumption_grid},
and
\ref{itm:stokes_discretisation:assumption_eta} hold.
Assume further that the weak solution
$(\uvec,\pscalar)$
to
\eqref{eq:governing_equations:momentum_balance} --
\eqref{eq:governing_equations:pressure}
belongs
to $\sobolev{n+1}{\domain}^2\times\sobolev{n}{\domain}$,
with $m\geq1 $. Choose the penalty values according to
\eqref{eq:stokes_discretisation:sipg_penalty} in lemma
\ref{thm:stokes_discretisation:discrete_coercivity}.
Let $\uhvec\in\spaceVhf$ and $\phscalar\in\spaceQhf$
denote the discrete solution
to problem \eqref{eq:stokes_discretisation:discrete_momentum_balance} --
\eqref{eq:stokes_discretisation:discrete_incompressible_constraint}.
Then, it holds that
\begin{align}
\normLebesgueTwo{\uvec-\uhvec}{\domain} &\leq
C\,h^{\min({k+1},n+1)}\,
\left(
\frac{\Cah^2}{\dgalpha}
\frac{1}{\dgbeta^2}
\normSobolev{\uvec}{n+1}{\domain} +
\Cah
\frac{1}{\dgbeta}
\normSobolev{\pscalar}{n}{\domain}
\right),
\end{align}
where the constant $C>0$ is independent of
$h$ and $k$ but depends on the shape regularity of the grid.
See Lemma \ref{thm:stokes_discretisation_error:error_norm_h} for
a definition of the remaining constants.
\end{lemma}
\begin{remark}
\rm We emphasise that the discretisation errors depend
on the size of the largest penalty parameter as well as on the viscosity
contrast by means of the constants $C_{a,h}$ and $\alpha_h$.
\end{remark}
\section{Preconditioning the Stokes system}
\label{sec:stokes_preconditioner}
The discrete variational problem
\eqref{eq:stokes_discretisation:discrete_momentum_balance}
-- \eqref{eq:stokes_discretisation:discrete_incompressible_constraint}
is equivalent to the saddle point problem
\begin{align}
\label{eq:discretestokes}
\begin{bmatrix}
\laMat{A}\phantom{\transp} & \laMat{B} \\
\laMat{B}\transp & \laMat{0}
\end{bmatrix}
\begin{bmatrix}
\laVec{u} \\
\laVec{p}
\end{bmatrix}
&=
\begin{bmatrix}
\laVec{f} \\
\laVec{g}
\end{bmatrix}.
\end{align}
We solve \eqref{eq:discretestokes} using a right preconditioned Krylov method,
with an upper block triangular preconditioner $\laMat {\mathcal P}$ of the form:
\begin{align}
\label{eq:discretestokespc}
\laMat{\mathcal P} = \begin{bmatrix}
\laMat A	&\laMat B \\
\laMat 0	&\laMat S
\end{bmatrix},
\end{align}
where $\laMat S = \laMat B^T \laMat A^{-1} \laMat B$ is the pressure Schur complement.

Our implementation employs tensor products of pairwise orthogonal
Legendre polynomials as basis functions for 
the pressure and the velocity components. 
The zero pressure average is not build into the pressure basis functions.
In case the Neumann boundary is empty, we thus solve a singular
system. As we typically use Krylov methods which are mathematically equivalent
to GMRES, we then require that the right-hand side is consistent
(e.g., we remove the constant pressure null space)
\cite[Theorem 2.4.]{brown_gmres_1997}.

Noting that both $\laMat{A}$ is symmetric positive definite (stemming from the SIP formulation) and $\laMat{S}$ is
symmetric positive definite (stemming from the inf-sup stability), this choice for $\laMat{\mathcal P}$ will result in convergence
in at most two iterations in exact arithmetic \cite{elma.iffis:2005}.
Whilst optimal (in the sense of iteration counts), the definition of
$\laMat{\mathcal P}$ is not practical as it involves an exact inverse for
$\laMat{A}$ and $\laMat{S}$.
A practical Stokes preconditioner replaces $\laMat{A}^{-1}$ and $\laMat{S}^{-1}$
by spectrally equivalent operators such that their application of the inverse on
a vector is significantly cheaper.

In our computations, we replace the Schur complement $\laMat{S}$ by the pressure
mass matrix scaled by the inverse of the element viscosity ($\laMat{S}^*$).
The proof of spectral equivalence between $\laMat{S}$ and $\laMat{S}^*$
for our DG spaces stems immediately from \cite[Theorem 5.22]{elma.iffis:2005}.
Due to the use of an orthonormal basis for the pressure space, $\laMat{S}^*$ is
diagonal.
The definition of $\laVec{y} = \laMat{A}^{-1} \laVec{x}$ is further replaced
with a preconditioned Krylov method with a fixed relative stopping condition which is described in detail in section
\ref{sec:stokes_preconditioner:viscous_block}.
\section{Preconditioning the viscous block}
\label{sec:stokes_preconditioner:viscous_block}
When used in conjunction with a Krylov method, we are required to apply the
action of $\laMat{\mathcal P}^{-1}$ on an arbitrary vector $\laVec{x} =
(\laVec{x}_u, \laVec{x}_p)$. We consider replacing the definition $\laVec{y}_u =
\laMat{A}^{-1} \laVec{x}_u$ with a spectrally equivalent operation:
solve $\laMat{A} \laVec{y}_u = \laVec{x}_u$ for $\laVec{y}_u$ using a
preconditioned Krylov method such that at the $i$-th iteration, $\| \laMat{A} \laVec{y}^i_u - \laVec{x}_u \| /
\| \laVec{x}_u \| < \epsilon$. In order to develop an optimal and scalable preconditioner for $\laMat{A}$, in this work we utilise a \textit{hp}-multilevel preconditioner.
The \textit{hp}-multilevel preconditioner employs coarsening with respect to
both the polynomial order $k$ of the velocity function space
(``$p$-coarsening'') and the spatial resolution $h$.
The rational for coarsening in both $p$ and $h$ will be elaborated below.

To introduce the $hp$-multilevel preconditioner, we first recall the basic twolevel multigrid method (see Algorithm
\ref{alg:mg}).
\begin{algorithm}
\caption{Twolevel Multigrid}\label{alg:mg}
\begin{algorithmic}[1]
\Procedure{MGVCycle}{$\laMat{A}$, $\laVec {f}$, $\laMat{\mathcal M}$,
$\laMat{R}$, $\laMat{P}$, $\laMat{\bar{A}}$, $\laVec{y}$}
\State{Set $i=0, \laVec{u}^0 = \laVec{y}$}
\Repeat
\State $\laVec{u}^i = \laVec{u}^i + \laMat{\mathcal M}^{-1}\,(\laVec{f} - \laMat A \, \laVec{u}^i
)$
\Comment{pre-smooth $m$ times}
\State $\laVec{\bar{r}} = \laMat{R}\,(\laVec{f} - \laMat
A \, \laVec{u}^i)$
\Comment{restrict residual}
\State Solve $\laMat{\bar{A}}\,\laVec{\bar{e}} = \laVec{\bar{r}}$
\Comment{solve
for the coarse grid correction} \State $\laVec{u}^i = \laVec{u}^i + \laMat{P}\,
\laVec{\bar{e}}$
\Comment{prolongate error} \State $\laVec{u}^i = \laVec{u}^i +
\laMat{\mathcal M}^{-\textup{T}} (\laVec{f} - \laMat{A}\,\laVec{u}^i )$
\Comment{post-smooth $m$ times}
\State $\laVec{u}^{i+1} = \laVec{u}^i$
\Comment{update for next iteration}
\State $i = i+1$
\Until{converged}
\State $\laVec{u} = \laVec{u}^i$
\EndProcedure
\end{algorithmic}
\end{algorithm}
The essential components of the multigrid algorithm are the fine level operator
$\laMat{A}$, the coarse level operator $\laMat{\bar{A}}$, the restriction and
prolongation operators $\laMat{R}$, $\laMat{P}$, which map vectors from the
fine level to the coarse level (and vice-versa), and the smoothing operator
$\laMat{\mathcal M}$.

In the context of $p$-multigrid, restriction refers to mapping a discrete vector defined using a
function space of order $k$ to a function space of order $s$, where $0 \le s <
k$,
Contrary to traditional $p$-coarsening strategies with hierarchies like $\spaceVhi{k} \rightarrow \spaceVhi{k/2} \rightarrow \spaceVhi{k/4} \dots$, here
we follow \cite{van_slingerland_fast_2014, van_slingerland_scalable_2015} and consider ``aggressive'' coarsening from $\spaceVhi{k}$ to an a-priori defined $p$-coarse space in a single step. 
This implicitly defines a twolevel hierarchy in $p$-space.
In this work, we study two different $p$-coarse spaces, namely
the space of element-wise constants $\spaceVhi{0}$ and
the space of $d$-linear functions $\spaceVhi{1}$.
We will use the symbols $\laMat{A}_0$ and $\laMat{A}_1$ for
the associated coarse grid operators.


The construction of the restriction and prolongation operators between different
order basis functions is natural to implement as we have employed a hierarchical basis. Furthermore, the prolongation operators are identical to the transposed restriction
operators. 
We define coarse operators via Galerkin projection.
Denoting the prolongation from polynomial degree $k$ to $s$ via $P_k^s$, the coarse level operators we consider are thus given by
$$
\laMat{A}_0 = \left( \laMat{P}_k^0 \right)\transp \laMat{A}\,\laMat{P}_k^0, 
\quad
\laMat{A}_1 = \left( \laMat{P}_k^1 \right)\transp \laMat{A}\,\laMat{P}_k^1.
$$

The smoother is defined as a Chebyshev iteration preconditioned
with a element-block Jacobi operator that consists of the diagonal blocks of $\laMat{A}$.
The minimum ($\lambda_0$) and maximum ($\lambda_1$) eigenvalue bounds required by Chebyshev are defined in the following manner. First, we estimate the maximum eigenvalue ($\lambda^*$) of $\laMat{A}$ by performing 10 iterations of GMRES with a random right hand side vector. We then choose $\lambda_0 = 0.1 \lambda^*$ and $\lambda_1 = 1.1 \lambda^*$ respectively. The choice of factors 0.1 and 1.1 have been determined empirically, however they are robust for variable coefficient scalar / vector elliptic problems and 
in fact are the default values used PETSc's Chebyshev implementation.

The bilinear form associated with the element-block Jacobi operator is:
\begin{align}
\label{eq:stokes_preconditioner:smoother}
\dgformAprec{\uhvec}{\vhvec} &=
\sumCells \Bigg (
\intCell 2\,\eta\,\strain{\uhvec} : \strain{\vhvec}\,\dV
\\
-
\sum_{\substack{\face\subset\cellBnd \\ \face\in\internalFaces}}
\intFace \eta^{\cell}\,\strain{\uhvec^{\cell}} : (\vhvec^{\cell} \otimes \normal^{\cell})\,\ds
&-
\sum_{\substack{\face\subset\cellBnd \\ \face\in\internalFaces}}
\intFace \eta^{\cell} \,\strain{\vhvec^{\cell}} :
(\uhvec^{\cell} \otimes \normal^{\cell})
\,\ds
\notag\\
+
\sum_{\substack{\face\subset\cellBnd \\ \face\in\internalFaces}}
\penalt\intFace
(\uhvec^{\cell}\otimes\normal^{\cell}) :
(\vhvec^{\cell}\otimes\normal^{\cell})\,\ds \notag
&+ \textup{ boundary terms} \Bigg), \notag
\end{align}
with $\uhvec,\vhvec\in\spaceVhf$, $\vhvec^{\cell}=\vhvec|_\cell$, and
$\normal^{\cell}$ denoting the outward normal to the boundary of element
$\cell\in\tria$. Notice the division by two in the second and third term stemming from the averaging on interior faces.

From similar arguments as in the proof of lemma
\ref{thm:stokes_discretisation:discrete_coercivity} follows
that the form \eqref{eq:stokes_preconditioner:smoother} is elliptic on
$\spaceVhf$.
Consequently, the element-block Jacobi operator
$\laMat{\mathcal{M}}_k$ is symmetric and positive-definite.
Following the proofs of \cite[Corollary 1 and
Equation (54)]{van_slingerland_scalable_2015}, one can then show that
\begin{align}
\kappa_2 \left( \laMat{\mathcal M}_k^{-1}\,\laMat{A}\right) \leq
1 +
C_1\,\max_{\cell\in\tria}\left\{\frac{\max_{\face\subset\cellBnd}\penal}{\min_{\textbf{x}\in
K}\eta} \right\}
+ C_2\,\frac{1}{\etaMin}
\end{align}
with constants $C_1>0$ and $C_2>0$ independent of $h$.
This result emphasises the importance of choosing the penalty parameters
based on local values of the viscosity.

\subsection{Preconditioning the coarse problem}
\label{sec:stokes_preconditioner:viscous_block:coarse}
In the context of high-resolution simulations, a pure $p$-multilevel preconditioner will never yield optimal $O(n)$ solve times
due to the increasing cost of performing the solve on the coarsest level.
This motivates us to employ a $h$-multigrid preconditioner for the coarse operators $\laMat{A}_0$ and 
$\laMat{A}_1$, respectively.
To realise this, we have adopted standard multigrid techniques which have been developed for 
finite difference discretisations and low-order finite discretisations. 
Below we elaborate on how these techniques are fused with our SIP-DG spatial discretisation.

\subsubsection{Element-wise constant coarse problem}
\label{sec:stokes_preconditioner:viscous_block:q0}
Let us denote  by $\mathcal{A}_{h.0}$ the
bilinear form inherited from $\mathcal{A}_{h}$ that corresponds to the element-wise
constants. Discretising $\mathcal{A}_{h.0}$ for an isoviscous fluid yields a
stencil which mimics a standard 5-point finite difference (FD) stencil
\cite{van_slingerland_fast_2014}.
Hence, heuristically it seems plausible to assume that any geometric multigrid preconditioner
suitable for a 5-point FD stencil should be appropriate to use as a preconditioner for the coarse grid solver associated
with $\laMat{A_0}$.

We first generate a hierarchy of meshes (with differing $h$) by isotropically coarsening the mesh defining $\laMat{A}_0$.
The maximum number of times
coarsening can be applied, and thus the number of levels in the $h$-multilevel preconditioner, is determined by the
spatial resolution of the grid. 
Note however that our implementation does not
support semi-coarsening, thus the finest grid must always employ an odd number of
elements in the $i$ and $j$ directions.
Between each mesh in the hierarchy, we have a restriction operator $\laMat{R}_h$ defined
by bilinear interpolation $(Q^2_1)$ and again we will use $\laMat{P}_h = \laMat{R}_h\transp$.

As in the $p$-multigrid implementation, we define coarse operators 
via Galerkin projection, e.g.
$\laMat{\bar{A}}_h = \laMat{P}_h\transp \laMat{A}_0\,\laMat{P}_h$.
The construction of Galerkin coarse operators is applied recursively for all levels in the mesh hierarchy. 
The smoother used within the $h$-multilevel
hierarchy is Chebyshev preconditioned with Jacobi. The Chebyshev bounds are
estimated similarly as for the $p$-coarsening smoother. 
On the coarsest level of the $h$-hierarchy, we will apply an exact LU factorisation. 
\subsubsection{Element-wise bilinear coarse problem}
\label{sec:stokes_preconditioner:viscous_block:q1}
In the case when the coarse space is $Q^2_1$, rather than leverage a finite difference analog to build a $h$-multigrid preconditioner, we will exploit methods designed for low order finite elements. 
Specifically, we consider projecting the $Q^2_1$ discontinuous space into the space of continuous bilinear functions.

We first define the discontinuous to continuous projector as the transpose
of the continuous to discontinuous projector.  The latter is a simple element-wise nodal to
modal projection. Let us denote the continuous-to-discontinuous
and discontinuous-to-continuous projectors by $\laMat{P}_\text{cd}$ and
$\laMat{R}_\text{dc}=\laMat{P}_\text{cd}^\text{T}$, respectively. 
We then project the discontinuous coarse level problem $\laMat{A}_{1}$ into $Q^2_1$ 
via $\laMat{A}_{1,\text{c}} = \laMat{P}_\text{cd}^\text{T}\,\laMat{A}_{1}\,\laMat{P}_\text{cd}$.
As in the element-wise constant case, a mesh hierarchy is created via isotropic coarsening 
and again we utilise linear interpolation and transposed restriction between each level.
All coarse operators are then constructed from $\laMat{A}_{1,\text{c}}$ and recursive application of Galerkin projection.
The same smoother and coarse level solver are used as in the element-wise constant case.

It is important to note that the size of the continuous $p$-coarse grid problem
$\laMat{A}_{1,\text{c}}$ equals the size of the element-wise constant coarse grid problem $\laMat{A}_\text{c}$
for comparable problem sizes.
For a $33\times33$ element mesh, e.g., the element-wise constant $p$-coarse space is
spanned by $33\times33\times2$ constants
while for a $32\times32$ element mesh, the continuous $p$-coarse space is spanned by 
$33\times33\times2$ bilinear (hat) functions.
\section{The heterogeneous viscosity Stokes benchmark SolCx}
\label{sec:solcx}
In order to verify the theoretical approximation properties of our
SIP based Stokes discretisation for heterogeneous problems,
we consider the \textit{SolCx} benchmark which has been considered
extensively for both solver and discretisation developments
\cite{may2008preconditioned, duretz_discretization_2011, kronbichler_high_2012}.
The analytic solution to the above problem is described in
\cite{zhong_analytic_1996} and is available as part of the \texttt{Underworld}
package \cite{moresi_computational_2007}.

Let $\domain$ be the unit square and let the viscosity $\eta$ contain a jump
in the lateral direction located along the line $0 < x_c < 1$.
We consider the problem of finding a solution to
\eqref{eq:governing_equations:momentum_balance} --
\eqref{eq:governing_equations:incompressibility_constraint}
s.t.~homogeneous Navier boundary conditions.
The free parameters of the model are chosen according to
$x_c=0.5$, $\eta_2=1$, and $\eta_1=10^6$.

Let us denote by $e_\uvec = \uvec - \uhvec$ and $e_\pscalar = \pscalar -
\phscalar$ the discretisation error of the velocity and the pressure fields,
respectively. The errors measured in the $\lebesgueTwo{\domain}$ norm for grids
employing an even number of elements in each direction and for different polynomial orders
are reported in Table~\ref{tab:solcx:even_grid}.
The discrete problem is solved to machine (double) precision.

For grids with an even number of elements in each coordinate direction, the jump
in viscosity is aligned with the edges of the elements. Consequently the discontinuous
basis functions accurately resolve the pressure field and thus optimal
convergence in $h$ is obtained.  
This was predicted in lemma \ref{thm:stokes_discretisation_error:error_norm_h}.
From the same lemma, we could only expect $\lebesgueTwo{\domain}$ convergence
in velocity that is suboptimal by one order however we observe optimal convergence rates.
Note that the $\lebesgueTwo{\domain}$ error estimate from lemma
\ref{thm:stokes_discretisation_error:error_norm_l2_uniform} is not applicable here
since the pressure solution is discontinuous.
The measured order of accuracy is shown in the final row
within Table~\ref{tab:solcx:even_grid}.

For grids employing $N^2$ elements, where $N$ is an odd number,
convergence degrades to first order convergence in velocity and
convergence by half an order in pressure (results not shown).
Degraded convergence was expected due to standard interpolation
error results.

\begin{table}
\centering
\caption[SolCx: Even grid]{SolCx Stokes benchmark. Results for grids
with a even number of elements in the $x$ and $y$ directions ($\Delta x$ = $\Delta y$).
All values below the horizontal lines in the second table
are considered as affected by machine precision.}
\label{tab:solcx:even_grid}
\renewcommand{\arraystretch}{1.1}
\footnotesize
\begin{tabular}{cccccccc}
\toprule
{} & \multicolumn{2}{c}{{$Q^2_1$--$Q_0$}} &
\multicolumn{2}{c}{$Q^2_2$--$Q_1$}  & \multicolumn{2}{c}{$Q^2_3$--$Q_2$}
\\
\cmidrule(l{2pt}r{2pt}){2-3}
\cmidrule(l{2pt}r{2pt}){4-5}
\cmidrule(l{2pt}r{2pt}){6-7}
$\Delta x$ &
\multicolumn{1}{c}{$\|e_\uvec\|_{L^2}$} &
\multicolumn{1}{c}{$\|e_\pscalar\|_{L^2}$}
&
\multicolumn{1}{c}{$\|e_\uvec\|_{L^2}$} &
\multicolumn{1}{c}{$\|e_\pscalar\|_{L^2}$}
&
\multicolumn{1}{c}{$\|e_\uvec\|_{L^2}$} &
\multicolumn{1}{c}{$\|e_\pscalar\|_{L^2}$}
\\
\midrule
1/2     & $1.3\times 10^{-3}$ & $6.7\times 10^{-2}$  & $6.5\times 10^{-4}$ & $1.4\times 10^{-2}$  & $9.9\times 10^{-5}$ & $2.0\times 10^{-3}$  \\
1/4     & $7.6\times 10^{-4}$ & $3.5\times 10^{-2}$  & $9.7\times 10^{-5}$ & $3.7\times 10^{-3}$  & $7.0\times 10^{-6}$ & $2.6\times 10^{-4}$  \\
1/8     & $2.2\times 10^{-4}$ & $1.7\times 10^{-2}$  & $1.2\times 10^{-5}$ & $9.4\times 10^{-4}$  & $4.5\times 10^{-7}$ & $3.2\times 10^{-5}$  \\
1/16    & $5.7\times 10^{-5}$ & $8.7\times 10^{-3}$  & $1.5\times 10^{-6}$ & $2.3\times 10^{-4}$  & $2.9\times 10^{-8}$ & $4.0\times 10^{-6}$  \\
1/32    & $1.4\times 10^{-5}$ & $4.4\times 10^{-3}$  & $1.9\times 10^{-7}$ & $5.9\times 10^{-5}$  & $1.8\times 10^{-9}$ & $5.1\times 10^{-7}$  \\
1/64    & $3.6\times 10^{-6}$ & $2.2\times 10^{-3}$  & $2.4\times 10^{-8}$ & $1.5\times 10^{-5}$  & $1.1\times 10^{-10}$ & $6.3\times 10^{-8}$  \\
1/128   & $9.1\times 10^{-7}$ & $1.1\times 10^{-3}$  & $3.0\times 10^{-9}$ & $3.7\times 10^{-6}$  & $7.0\times 10^{-12}$ & $7.9\times 10^{-9}$  \\
\cmidrule(l{2pt}r{2pt}){2-3}
\cmidrule(l{2pt}r{2pt}){4-5}
\cmidrule(l{2pt}r{2pt}){6-7}
&$\mathcal{O}(h^{1.98})$&$\mathcal{O}(h^{1.00})$&$\mathcal{O}(h^{3.00})$&$\mathcal{O}(h^{2.02})$&$\mathcal{O}(h^{3.97})$&$\mathcal{O}(h^{3.00})$\vspace{5mm}\\
\end{tabular}
\begin{tabular}{cccccccc}
\toprule
{} & \multicolumn{2}{c}{$Q^2_4$--$Q_3$} &
\multicolumn{2}{c}{$Q^2_5$--$Q_4$}  & \multicolumn{2}{c}{$Q^2_6$--$Q_5$}
\\
\cmidrule(l{2pt}r{2pt}){2-3}
\cmidrule(l{2pt}r{2pt}){4-5}
\cmidrule(l{2pt}r{2pt}){6-7}
$\Delta x$ &
\multicolumn{1}{c}{$\|e_\uvec\|_{L^2}$} &
\multicolumn{1}{c}{$\|e_\pscalar\|_{L^2}$}
&
\multicolumn{1}{c}{$\|e_\uvec\|_{L^2}$} &
\multicolumn{1}{c}{$\|e_\pscalar\|_{L^2}$}
&
\multicolumn{1}{c}{$\|e_\uvec\|_{L^2}$} &
\multicolumn{1}{c}{$\|e_\pscalar\|_{L^2}$}
\\
\midrule
1/2   &$7.1\times 10^{-6}$&$2.1\times 10^{-4}$ &$4.8\times 10^{-7}$&$1.4\times 10^{-5}$ &$3.7\times 10^{-8}$&$9.7\times 10^{-7}$  \\
1/4   &$2.5\times 10^{-7}$&$1.3\times 10^{-5}$ &$9.4\times 10^{-9}$&$4.5\times 10^{-7}$ &$3.4\times 10^{-10}$&$1.6\times 10^{-8}$  \\
\cline{6-7}
1/8   &$8.2\times 10^{-9}$&$8.4\times 10^{-7}$ &$1.6\times 10^{-10}$&$1.4\times 10^{-8}$ &$5.4\times 10^{-12}$&$2.8\times 10^{-10}$  \\
\cline{4-5}
1/16  &$2.6\times 10^{-10}$&$5.3\times 10^{-8}$ &$4.8\times 10^{-12}$&$4.8\times
10^{-10}$ &&  \\
1/32  &$8.3\times 10^{-12}$&$3.3\times 10^{-9}$ & & & \\
\cline{2-3}
1/64  &$1.8\times 10^{-12}$ & $2.2\times 10^{-10}$ & & & &   \\
\cmidrule(l{2pt}r{2pt}){2-3}
\cmidrule(l{2pt}r{2pt}){4-5}
\cmidrule(l{2pt}r{2pt}){6-7}
&$\mathcal{O}(h^{4.96})$&$\mathcal{O}(h^{4.01})$&$\mathcal{O}(h^{5.88})$&$\mathcal{O}(h^{5.01})$&$\mathcal{O}(h^{6.77})$&$\mathcal{O}(h^{5.92})$\\
\end{tabular}
\end{table}$~$
\section{Solver performance} \label{sec:solver_performance}
In this section, we evaluate the robustness and scalability of
the Stokes solver discussed in Sec.~\ref{sec:stokes_preconditioner:viscous_block}.
Specifically we consider four variants of the preconditioner associated with the
$\laMat{A}$ operator described in Sec.~\ref{sec:stokes_preconditioner:viscous_block}.
The first two configurations $\mathcal{A}_p$($Q^2_0$) and
$\mathcal{A}_{hp}(Q^2_0)$ employ a two level $p$-coarsening strategy in which the polynomial order is 
aggressively coarsened until we obtain a $Q^2_0$ basis.
The latter additionally applies geometric coarsening as detailed
in section \ref{sec:stokes_preconditioner:viscous_block:q0}.
Similarly we introduce 
$\mathcal{A}_p$($Q^2_1$) and $\mathcal{A}_{hp}(Q^2_1)$ where
the latter uses a geometric coarsening strategy as outlined in 
\ref{sec:stokes_preconditioner:viscous_block:q1}.
The number of Chebyshev-accelerated element-block Jacobi smoothing steps for
the two level $p-$coarsening is set to 2 for all preconditioners (up and down
smoothing each).
The smoothers employed in the $h$-coarsening part
of $\mathcal{A}_{hp}(Q^2_0)$ and $\mathcal{A}_{hp}(Q^2_1)$
run 3 iterations (up and down smoothing each). 
In all experiments, the coarsest level in both the $p$-multigrid, and $hp$-multigrid variants employed LU factorisation.

All numerical experiments were performed on a
single node of ``Hamilton'', located at Durham University (UK), equipped with
two Intel Xeon E5-2650 v2 (Ivy Bridge) 8 core 2.6 GHz processors with
64 GB of RAM. Experiments that state solve times have been
performed using only a single processor. If not other otherwise indicated, we
will perform experiments in double precision.
\subsection{SolCx}
We consider the SolCx benchmark
(Sec.~\ref{sec:solcx}) with parameters $x_c=0.5$, $\eta_2=1$. 
As free parameters we use the viscosity contrast $\Delta\eta=\eta_1:\eta_2$, 
and the grid resolution in our tests.
We discretise the problem using a second order velocity space
and a first order pressure space ($Q^2_2$--$Q_1$ elements).
For preconditioners $\mathcal{A}_{p}$($Q^2_0$), $\mathcal{A}_{p}$($Q^2_1$), 
and $\mathcal{A}_{hp}(Q^2_1)$, we consider meshes with the sizes
$64^2$, $128^2$, $256^2$,and $512^2$.
We remark that we have to use odd numbers of elements in each
coordinate direction if we want to use the preconditioner
$\mathcal{A}_{hp}(Q_0$); see Sec.
\ref{sec:stokes_preconditioner:viscous_block:q0}.
We then consider the mesh with element resolutions of  
$65^2$, $129^2$, $257^2$,and $513^2$.
We further perform an $L^2$ projection of the
original viscosity on the element-wise constants since the mesh does not align
with the viscosity structure for these meshes.
For both preconditioners $\mathcal{A}_{hp}(Q^2_0)$ and
$\mathcal{A}_{hp}(Q^2_1)$, we use 3, 4, 5, and 6 $h$-multigrid levels for
the considered mesh sizes, respectively.

We used right preconditioned FGMRES to solve the Stokes problem.
Convergence of the saddle point problem is deemed to have occurred when the 2-norm of the residual is $10^6$ times smaller than the initial residual (which we denote via ${\mathrm{rtol(FGMRES)}=10^{-6}}$).
The inner solver applied to $\laMat{A}$ is preconditioned CG and is terminated
according to a relative tolerance criterion of $\mathrm{rtol(CG)}=10^{-3}$.
A flexible Krylov method is not required for the viscous block solve since the $p$- and $hp$-multigrid 
preconditioners are linear operators.
The wall-clock time and iterations required to solve the Stokes problem,
as well as the iterations required by the viscous block solve are reported
in Table~\ref{tab:solver_performance:stokes:solcx}.

The overall Stokes solver is observed to be scalable
for all four preconditioners as the number of outer and inner iterations are
virtually independent of the grid resolution for a given viscosity contrast.
The viscous block preconditioners $\mathcal{A}_{p}(Q^2_1)$ and 
$\mathcal{A}_{hp}(Q^2_1)$ yield significantly
less inner iterations than the other two preconditioners. 

We note that due to the
usage of an LU factorisation on the coarsest level problem, we do not observe optimal (e.g.
$O(n)$) solve times for the variants which do not employ
$hp$-multigrid. Clearly, $\mathcal{A}_{p}(Q^2_1)$ has a significantly
larger coarse grid problem than the other solvers; (no
discontinuous-to-continuous projection is performed), thus CPU time is far from optimal.
We note that the timings for $\mathcal{A}_{p}(Q^2_0)$ are close to optimal. 
However, for increasing problem sizes, can can expect further departure from $O(n)$ as observed when using $\mathcal{A}_{p}(Q^2_1)$.
Variant $\mathcal{A}_{hp}(Q^2_1)$ clearly outperforms the three other
preconditioners in terms of solve time, and the 
inner solver only needs one more iterations 
(at maximum) than $\mathcal{A}_{p}(Q^2_1)$. 

\begin{table}
\renewcommand{\arraystretch}{1.1}
\footnotesize
\caption{SolCx: Performance of the Stokes solver using
different viscous block preconditioners (see text for details) 
as a function viscosity jump $(\Delta \eta)$ and the number of
elements $N_K$. The polynomial order of the velocity is fixed to $k=2$.
Here, \#it indicates the number of outer iterations applied to the Stokes operator,
whilst numbers in brackets indicate the average and maximum iterations required
by the viscous block solver.
$t$ denotes the CPU time required for the solve.
See the text for details on the
stopping criteria.}
\label{tab:solver_performance:stokes:solcx}
\centering
\begin{tabular}{cp{0.5cm} r rr rr}
\toprule
&{} & \multicolumn{2}{c}{$\Delta \eta = 10^0$} &
\multicolumn{2}{c}{$\Delta \eta = 10^6$}
\\
\cmidrule(l{2pt}r{2pt}){3-4}
\cmidrule(l{2pt}r{2pt}){5-6}
Preconditioner for $\laMat{A}$  &
$N_K$ &
\multicolumn{1}{c}{\#it} &
\multicolumn{1}{c}{$t \,(\textup{s})$}
&
\multicolumn{1}{c}{\#it} &
\multicolumn{1}{c}{$t \,(\textup{s})$}
\\
\midrule
{$\mathcal{A}_{p}(Q^2_0)$} 
         & $64^2$   & 3 (21.3, 23) & $2.76$   & 5 (34.2, 43) & $12.23$\\
{}       & $128^2$  & 3 (20.7, 22) & $12.24$  & 5 (36.6, 45) & $54.19$\\
{}       & $256^2$  & 3 (20.7, 23) & $54.71$  & 5 (35.4, 46) & $212.55$\\
{}       & $512^2$ & 3 (20.0, 23) & $302.21$ & 6 (33.5, 46) & $965.76$\\
\midrule
{$\mathcal{A}_{hp}(Q^2_0)$} 
         & $65^2$   & 3 (21.3, 23) & $4.90$   & 5 (36.6, 45) & $13.61$\\
{}       & $129^2$  & 3 (21.0, 22) & $19.26$  & 5 (39.0, 46) & $57.74$\\
{}       & $257^2$  & 3 (20.7, 22) & $75.43$  & 5 (40.4, 46) & $237.43$\\
{}       & $513^2$ & 3 (20.7, 23) & $301.30$ & 5 (43.8, 54) & $1028.10$\\
\midrule
{$\mathcal{A}_{p}(Q^2_1)$}  
         & $64^2$   & 3 (3.3, 4) & $1.47$   & 5 (4.8, 6) & $3.55$   \\
{}       & $128^2$  & 3 (3.3, 4) & $21.74$   & 5 (4.8, 6) & $30.87$  \\
{}       & $256^2$  & 3 (3.3, 4) & $323.77$  & 5 (4.6, 6) & $355.21$  \\
{}       & $512^2$  & 3 (3.3, 4) & $2681.80$  & 5 (4.6, 6) & $2805.20$ \\
\midrule
{$\mathcal{A}_{hp}(Q^2_1)$}  
         & $64^2$   & 3 (4.0, 4) & $1.14$   & 5 (5.2, 7) & $2.22$   \\
{}       & $128^2$  & 3 (3.7, 4) & $4.29$   & 5 (5.4, 7) & $10.75$  \\
{}       & $256^2$  & 3 (3.7, 4) & $17.66$  & 5 (5.2, 6) & $37.18$  \\
{}       & $512^2$  & 4 (3.7, 4) & $70.59$  & 5 (5.4, 7) & $155.52$ \\
\bottomrule
\end{tabular}
\end{table}
\subsection{SolCx checkerboard}
In the last section we observed that the preconditioners based on an 
element-wise bilinear $p$-coarse space are significantly more
efficient in terms of iterations than their counterparts using an element-wise
constant coarse space.
The preconditioner $\mathcal{A}_{hp}(Q^2_1)$ applying $h$-coarsening was further
found to yield solve times that scale optimally and which are significantly
smaller than those of the other three preconditioners.
In the following tests, we will thus only consider $\mathcal{A}_{hp}(Q^2_1)$.

In order to demonstrate the robustness of this preconditioner for
harder problems, we again solve a SolCx setting but this time with an
additional viscosity jump in the $y$--direction at $y=0.5$. The resulting
viscosity structure is a $2\times2$ checkerboard. All solver components are
configured as detailed in the previous section.

This time we further investigate the influence of the polynomial order
on the convergence, and we consider the viscosity
contrasts $\eta_2 : \eta_1 = \Delta\eta\in\{10^3, 10^6, 10^8\}$.

The wall-clock time and iterations required to solve the Stokes problem,
as well as the iterations required by the viscous block solve are reported
in Table~\ref{tab:solver_performance:stokes:solcx_checkerboard}.

Incrementing the polynomial order by one increments the 
iterations necessary to converge the viscous solve 
by around 5 independent of the viscosity jump.  
A slight dependence of the outer iterations on the polynomial order $k$
is observed for this problem. The most noticeable dependence of
the outer iterations on $k$ can be observed between the elements
$Q^2_1$--$Q_0$ and $Q^2_2$--$Q_1$.

For a given $Q^2_k$--$Q_{k-1}$ element pair, the following additional observations can be made:
(1) close to optimal solve times are observed under mesh refinement, and 
(2) the solve times are almost independent of the jump in viscosity.

\begin{table}
\renewcommand{\arraystretch}{1.1}
\footnotesize
\caption{SolCx checkerboard: Performance of the Stokes solver using
configuration $\mathcal{A}_{hp}(Q^2_1)$ (see text for details) as a function of element order, viscosity jump $(\Delta \eta)$ and the number of elements $N_K$.
Refer to Table~\ref{tab:solver_performance:stokes:solcx} for the definition of
the data reported.
Columns reporting iterations which are marked with a ($*$) indicate jobs which required $>64$ GB of RAM and thus could not be executed.
}
\label{tab:solver_performance:stokes:solcx_checkerboard}
\centering
\begin{tabular}{p{0.2cm}p{0.5cm} rrr rrr}
\toprule
&{} & \multicolumn{2}{c}{$Q^2_1$--$Q_0$} &
\multicolumn{2}{c}{$Q^2_2$--$Q_1$}  & \multicolumn{2}{c}{$Q^2_3$--$Q_2$}
\\
\cmidrule(l{2pt}r{2pt}){3-4}
\cmidrule(l{2pt}r{2pt}){5-6}
\cmidrule(l{2pt}r{2pt}){7-8}
$\Delta\eta$ &
$N_K$ &
\multicolumn{1}{c}{\#it} &
\multicolumn{1}{c}{$t \,(\textup{s})$}
&
\multicolumn{1}{c}{\#it} &
\multicolumn{1}{c}{$t \,(\textup{s})$}
&
\multicolumn{1}{c}{\#it} &
\multicolumn{1}{c}{$t \,(\textup{s})$}
\\
\midrule
 {$10^3$} & $64^2$   & 14 (4.0, 5) & $0.97$             & 17 (7.9, 10)    &
 $10.47$      & 18 (11.7, 15) & $47.50$   \\
 {} & $128^2$  & 15 (4.2, 5) & $5.38$             & 17 (8.1, 11)    & $43.35$      & 18 (11.4, 16) & $187.51$  \\
 {} & $256^2$  & 16 (4.9, 6) & $27.20$            & 18 (8.2, 11)    & $188.94$     & 22 (11.0, 17) & $892.24$  \\
 {} & $512^2$  & 15 (5.1, 7) & $108.45$           & 17 (8.4, 12)    & $731.49$     & 23 (11.0, 18) & $3697.00$ \\
 {} & $1024^2$ & 15 (5.1, 7) & $458.02$           &*  &  &*  & \\
\midrule
 {$10^6$} & $64^2$   & 15 (4.1, 5) & $1.08$             & 17 (8.4, 11) & $11.15$ 
 & 17 (12.4, 16) & $47.85$   \\
 {} & $128^2$  & 19 (4.6, 6) & $7.27$             & 17 (8.3, 11) & $44.60$  & 17 (12.4, 17) & $193.43$  \\
 {} & $256^2$  & 15 (4.7, 6) & $24.79$            & 17 (8.6, 12) & $187.26$ & 20 (11.3, 17) & $833.87$  \\
 {} & $512^2$  & 17 (5.2, 7) & $122.68$           & 19 (8.7, 12) & $847.84$ & 19 (12.4, 18) & $3846.70$ \\
 {} & $1024^2$ & 17 (5.4, 8) & $508.84$           &*  &  &*  & \\
\midrule
 {$10^8$} & $64^2$   & 15 (4.2, 6) & $1.10$             & 16 (9.0, 11) & $11.20$ 
 & 18 (13.1, 17) & $52.84$   \\
 {} & $128^2$  & 15 (4.9, 6) & $8.03$             & 16 (9.2, 12) & $46.58$  & 18 (13.3, 18) & $218.42$  \\
 {} & $256^2$  & 14 (5.1, 6) & $24.69$            & 17 (9.3, 12) & $201.54$ & 19 (13.3, 18) & $918.59$  \\
 {} & $512^2$  & 19 (5.2, 7) & $136.65$           & 17 (9.6, 13) & $829.67$ & 18 (13.5, 19) & $3795.30$ \\
 {} & $1024^2$ & 21 (5.6, 8) & $814.33$           & *  &  &*  &  \\
\bottomrule
\end{tabular}
\end{table}
\subsection{Sedimenting viscous circular inclusions}
In our final test, we place six circular inclusions with viscosity
$\eta_2 \in\{10^3, 10^6, 10^8\}$ and density $\rho_2=1.2$ in a
medium with viscosity $\eta_1=1$ and density $\rho_1=1$.
We then consider a forcing term $\vec{f}=(0,-g\,\rho)$,
where the gravity constant is chosen as $g=10$.
The inclusions are placed at 
$(0.84,0.39)$, $(0.79,0.91)$, $(0.33,0.76)$,
$(0.55,0.47)$, $(0.14,0.60)$, $(0.24,0.13)$
and have radii
$0.089$, $0.059$, $0.063$, $0.081$, $0.05$ and $0.09$ respectively.
The model configuration is such that the inclusions sediment under gravity.
At the top boundary, homogeneous Neumann boundary conditions are imposed
while on the remaining parts of the boundary, homogeneous 
Navier boundary conditions are imposed.

An approximation of the model and a corresponding numerical velocity solution 
computed using $Q^2_2$--$Q_1$ discontinuous finite elements is
depicted in Fig.~\ref{fig:sinkers}.
The model is approximated using a uniform mesh consisting of 
$256^2$ elements.

\begin{figure}[htp]
\centering
\epsfig{width=0.55\textwidth,file=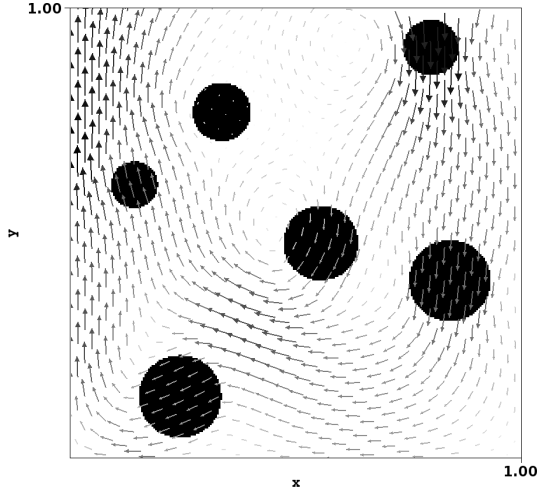}
\caption[]{Sedimenting circular inclusions: Approximate material composition
(pseudocolour plot) and a corresponding numerical velocity solution
(vector plot) computed using discontinuous $Q^2_2$--$Q_1$ elements. The model
is discretised using a uniform mesh consisting of $256^2$ elements.
Viscosity and density are chosen $\eta_1=1$ and $\rho_1=1$, respectively, for
the background material while they are chosen $\eta_2=10^6$ and $\rho_2=1.2$ for the
circular inclusions.}
\label{fig:sinkers}
\end{figure}

A Cartesian grid can never resolve these viscosity and
density distributions exactly, 
we thus always take the element-wise maximum
viscosity and the element-wise minimum density.
Thus, the discrete problem has by construction a mesh 
dependence.

We consider meshes with element resolutions of 
$64^2,\,128^2,\,256^2,\,512$, and $1024^2$ and employ
$2$, $3$, $4$, $5$ and $6$ $h$-coarsening levels; one level less
than in the previous experiments.
We use GCR for both the Stokes (outer) problem and the viscous block (inner) problem.

The wall-clock time and iterations required to solve the Stokes problem,
as well as the iterations required by the viscous block solve are reported
in Table~\ref{tab:solver_performance:stokes:sinker6}.
We vary again the polynomial order $k\in{1,2,3}$ and
the viscosity contrast $\Delta\eta\in\{10^3, 10^6, 10^8\}$.
As in the previous two experiments which used a simpler viscosity structure, using $\mathcal{A}_{hp}(Q^2_1)$ 
with the sinker model configuration we observe that the Stokes preconditioner is both scalable and near optimal. 
Outer iterations are observed to be only very weakly dependent on the jump in viscosity, and mildly dependent on the polynomial degree.
The inner iterations required to converge the viscous block are approximately independent of the viscosity jump for each polynomial degree considered.
As per other experiments, the average number of iterates required to converge the viscous block are mildly dependent on the polynomial degree.

\begin{table}
\renewcommand{\arraystretch}{1.1}
\footnotesize
\caption{Six circular inclusions: Performance of the Stokes solver using
configuration $\mathcal{A}_{hp}(Q^2_1)$ (see text for details) 
as a function of element order, viscosity jump $(\Delta \eta)$ and the number of elements $N_K$.
Refer to Table~\ref{tab:solver_performance:stokes:solcx} for the definition of
the data reported.}
\label{tab:solver_performance:stokes:sinker6}
\centering
\begin{tabular}{p{0.2cm}p{0.5cm} rrr rrr}
\toprule
&{} & \multicolumn{2}{c}{$Q^2_1$--$Q_0$} &
\multicolumn{2}{c}{$Q^2_2$--$Q_1$}  & \multicolumn{2}{c}{$Q^2_3$--$Q_2$}
\\
\cmidrule(l{2pt}r{2pt}){3-4}
\cmidrule(l{2pt}r{2pt}){5-6}
\cmidrule(l{2pt}r{2pt}){7-8}
$\Delta\eta$ &
$N_K$ &
\multicolumn{1}{c}{\#it} &
\multicolumn{1}{c}{$t \,(\textup{s})$}
&
\multicolumn{1}{c}{\#it} &
\multicolumn{1}{c}{$t \,(\textup{s})$}
&
\multicolumn{1}{c}{\#it} &
\multicolumn{1}{c}{$t \,(\textup{s})$}
\\
\toprule
{$10^3$} & $64^2$   & 17 (3.3, 6) & $0.97$   & 20 (6.7, 11) & $9.54$   &
21 (8.9, 16)      & $38.12$      \\
{} & $128^2$  & 16 (3.5, 7) & $4.47$   & 19 (6.6, 12) & $36.36$  &
22 (9.0, 16)      & $163.31$     \\
{} &$256^2$  & 17 (3.8, 8) & $21.16$  & 19 (6.8, 12) & $151.00$ &
22 (8.9, 17)      & $650.42$     \\
{} & $512^2$  & 19 (3.8, 9) & $96.98$ & 20 (6.7, 12) & $668.22$ & 26 (9.4, 35) & $3502.30$ \\
{} &$1024^2$ & 25 (4.0, 9) & $589.39$ & $*$ &  & $*$  & \\
\midrule
 {$10^6$} & $64^2$   & 17 (3.4, \phantom{0}6)  & $0.98$ & 20 (6.7,
 11) & $9.50$ & 20 (\phantom{0}8.9, 16)  & $36.28$   \\
 {} & $128^2$  & 17 (3.6, \phantom{0}8)  & $4.92$ & 18 (6.8, 12) &
 $35.47$ & 21 (\phantom{0}8.7, 16)  & $151.37$  \\
 {} & $256^2$  & 19 (4.1, 10) & $25.43$            & 18 (6.9, 13) & $145.18$ &
 24 (\phantom{0}9.0, 17)  & $715.33$  \\
 {} & $512^2$  & 19 (4.4, 11) & $108.61$           & 21 (6.9, 12) & $710.49$ & 25 (10.2, 35) & $4606.20$ \\
 {} & $1024^2$ & 22 (4.6, 11) & $603.12$           &  $*$  &  &  $*$ &  \\
\bottomrule
\end{tabular}
\end{table}

\subsubsection{Quadruple-precision floating point arithmetic}
\label{sec:results:sinkers_quad}
GCR was adopted in the previous experiment as we observed
that the orthogonalisation procedures of CG, GMRES, and
FGMRES would break-down for viscosity jumps $\Delta\eta > 10^3$
at a given mesh size.
We further note that this break-down behaviour appears to be independent of the viscous 
block preconditioner as it also occurred when using $\mathcal{A}_{p}(Q^2_1)$. 

We could trace the break-down back to being related to a lose of floating point precision. 
Using quadruple-precision floating point arithmetic, break-down of the orthogonalisation does not occur and 
our preconditioner is able to solve problems 
with extreme viscosity contrasts ($\Delta \eta \sim 10^{20}$).
A selected number of results using $\Delta \eta = 10^{10}, 10^{20}$ are reported in Table \ref{tab:solver_performance:stokes:sinker6_quad}.
As per the results obtained with double-precision, the Stokes and viscous block preconditioner are observed to be scalable, 
and solve times are close to optimal. 
For these experiments, the outer solver was chosen as FGMRES  
and the inner solver used was CG.

\begin{table}
\renewcommand{\arraystretch}{1.1}
\footnotesize
\caption{Six circular inclusions (quad precision): Performance of the Stokes
solver using configuration $\mathcal{A}_{hp}(Q^2_1)$ (see text for details) 
as a function of element order, extreme viscosity jumps $(\Delta \eta)$, and
the number of elements $N_K$.
Refer to Table~\ref{tab:solver_performance:stokes:solcx} for the definition of
the data reported. CPU time columns marked with (--) indicate the job required
longer than 24 hrs to complete.}
\label{tab:solver_performance:stokes:sinker6_quad}
\centering
\begin{tabular}{p{0.2cm}p{0.5cm} rrr rrr}
\toprule
&{} & \multicolumn{2}{c}{$Q^2_1$--$Q_0$} &
\multicolumn{2}{c}{$Q^2_2$--$Q_1$}  & \multicolumn{2}{c}{$Q^2_3$--$Q_2$}
\\
\cmidrule(l{2pt}r{2pt}){3-4}
\cmidrule(l{2pt}r{2pt}){5-6}
\cmidrule(l{2pt}r{2pt}){7-8}
$\Delta\eta$ &
$N_K$ &
\multicolumn{1}{c}{\#it} &
\multicolumn{1}{c}{$t \,(\textup{s})$}
&
\multicolumn{1}{c}{\#it} &
\multicolumn{1}{c}{$t \,(\textup{s})$}
&
\multicolumn{1}{c}{\#it} &
\multicolumn{1}{c}{$t \,(\textup{s})$}
\\
\toprule
 {$10^{10}$} 
    & $64^2$   & 17 (4.1, \phantom{0}7)  & $57.32$             & 18 (10.1, 15) &
    $571.04$  & 18 (13.3, 20) & $2154.60$  \\
 {} & $128^2$  & 16 (4.6, \phantom{0}9)  & $236.45$            & 17
 (\phantom{0}9.9, 17)  & $2087.30$  & 18 (13.9, 21) & $9153.80$  \\
 {} & $256^2$  & 18 (5.6, 11) & $1236.20$           & 17 (\phantom{0}8.8, 15)  &
 $7325.30$  & 18 (14.7, 23) & $37673.00$ \\
 {} & $512^2$  & 16 (6.1, 12) & $4693.20$           & 17 (\phantom{0}8.4, 16)  &
 $27888.00$ & -          & -          \\
\midrule
 {$10^{20}$} 
     & $64^2$   & 18 (8.3, 10)  & $118.30$            & 21 (14.9, 19) & $984.84$  
     & 26 (17.8, 26) & $4257.80$  \\
 {} & $128^2$  & 18 (10.7, 13) & $596.19$            & 18 (17.6, 21) & $4088.10$  & 25 (19.6, 27) & $18110.00$ \\
 {} & $256^2$  & 18 (12.3, 16) & $2697.60$           & 19 (15.0, 19) & $14058.00$ & 28 (17.7, 27) & $78188.00$ \\
 {} & $512^2$  & 18 (13.4, 17) & $11480.00$          & 16 (14.8, 20) & $60247.00$ & -          & -          \\
\bottomrule
\end{tabular}
\end{table}

\section{Conclusions} \label{sec:conclusion}
We have investigated high order SIP based discretisations of the variable
viscosity Stokes flow. We have demonstrated that the discretisations are
optimally convergent in $h$ for prototypical geodynamics problems 
where the viscosity discontinuity can be resolved by the grid.

For the solution of the saddle point system arising from the 
discretisation of the Stokes equations, we proposed an iterative
method based on block preconditioned FGMRES for the overall linear system and 
$hp$-multilevel preconditioned CG for the viscous block. 
We considered coarsening the polynomial degree of the viscous block to either the space of 
piecewise constants ($Q^2_0$), or bilinear functions ($Q^2_1$), 
and for each coarse space, a $h$-multigrid preconditioner was proposed.

Through a series of numerical experiments with heterogeneous viscosity, we
have demonstrated that the $h$-multigrid strategy results in a more robust coarse level preconditioner.
This was attributed to the fact that the $Q^2_0$ coarse space by construction excludes cross derivatives 
which appear in the definition of the stress tensor when the viscosity is a function of space. 
Neglecting these terms in the coarse space does not result in error corrections which drive the fine level residual to zero.
In contrast, the $h$-multigrid variant considered for the $Q^2_1$ coarse space problem results in a solver with a convergence rate that was observed to be independent of the
number of elements, largely insensitive to the both the viscosity structure and the jump in viscosity,
and only weakly dependent on the approximation order.

Lastly, we have outlined the importance of choosing the face-wise SIP penalty
parameters depending on the local viscosity and close to the lower bound of the stable regime in order to
minimise discretisation errors and the number of iterations of the nested
inner solver.

\subsection{Outlook}
Subject of future research could be an extension of the methodology to $h$- and
$p$-adaptive methods, to three dimensions, as well as to distributed and
shared memory parallelism.
Furthermore, it might be interesting to perform an analysis of the 
$p$-coarse level operator defined on the space of continuous, element-wise
bilinear functions. The operator might be related to a Nitsche type
discretisation.
In the context of extreme viscosity contrast problems ($\Delta \eta > 10^6$),
one might want to analyse for which substeps and constituents of the considered
preconditioned Krylov methods high precision is required.

\section{Acknowledgments}
All numerical experiments were performed using the PETSc library
\cite{petsc-user-ref,petsc-web-page,petsc-efficient}.
ETH Z\"urich is thanked for compute time on the Brutus and
Euler clusters. The Swiss National Supercomputing Centre (CSCS) is thanked for
compute time on Piz Daint.
This work made use of the facilities of the Hamilton HPC Service of Durham
University.

We thank two anonymous reviewers for their critical remarks which
motivated us to significantly improve our solver methodology.


\begin{thebibliography}{10}

\bibitem{arnold_sipg_1982}
{\sc D.~{A}rnold}, {\em {{A}n interior penalty finite element method with
  discontinuous elements}}, SIAM Journal on Numerical Analysis, 19 (1982),
  pp.~742--760.

\bibitem{arnold_unified_2002}
{\sc D.~{A}rnold, F.~{B}rezzi, B.~{C}ockburn, and L.~{M}arini}, {\em {{U}nified
  analysis of discontinuous {G}alerkin methods for elliptic problems}}, SIAM
  Journal on Numerical Analysis, 39 (2002), pp.~1749--1779.

\bibitem{ayuso_de_dios_simple_2014}
{\sc B.~{A}yuso~de {D}ios, F.~{B}rezzi, L.~D. {M}arini, J.~{X}u, and
  L.~{Z}ikatanov}, {\em A {simple} {preconditioner} for a {discontinuous}
  {{G}alerkin} {method} for the {{S}tokes} {problem}}, {J}ournal of Scientific
  Computing, 58 (2014), pp.~517--547.

\bibitem{petsc-user-ref}
{\sc S.~Balay, S.~Abhyankar, M.~F. Adams, J.~Brown, P.~Brune, K.~Buschelman,
  L.~Dalcin, V.~Eijkhout, W.~D. Gropp, D.~Kaushik, M.~G. Knepley, L.~C.
  McInnes, K.~Rupp, B.~F. Smith, S.~Zampini, H.~Zhang, and H.~Zhang}, {\em
  {PETS}c users manual}, Tech. Rep. ANL-95/11 - Revision 3.7, Argonne National
  Laboratory, 2016.

\bibitem{petsc-web-page}
\leavevmode\vrule height 2pt depth -1.6pt width 23pt, {\em {PETS}c {W}eb page}.
\newblock \url{http://www.mcs.anl.gov/petsc}, 2016.

\bibitem{petsc-efficient}
{\sc S.~{B}alay, W.~D. {G}ropp, L.~C. {M}c{I}nnes, and B.~F. {S}mith}, {\em
  {M}odern {S}oftware {T}ools in {S}cientific {C}omputing}, in {E}fficient
  {M}anagement of {P}arallelism in {O}bject {O}riented {N}umerical {S}oftware
  {L}ibraries, E.~{A}rge, A.~M. {B}ruaset, and H.~P. {L}angtangen, eds.,
  {B}irkh{\"a}user Press, 1997, pp.~163--202.

\bibitem{bastian_fully-coupled_2014}
{\sc P.~Bastian}, {\em A fully-coupled discontinuous {Galerkin} method for
  two-phase flow in porous media with discontinuous capillary pressure},
  Computational Geosciences, 18 (2014), pp.~779--796.

\bibitem{bastian_algebraic_2012}
{\sc P.~Bastian, M.~Blatt, and R.~Scheichl}, {\em Algebraic multigrid for
  discontinuous {Galerkin} discretizations of heterogeneous elliptic problems:
  {AMG}4dg}, Numerical Linear Algebra with Applications, 19 (2012),
  pp.~367--388.

\bibitem{braun2008douar}
{\sc J.~Braun, C.~Thieulot, P.~Fullsack, M.~DeKool, C.~Beaumont, and
  R.~Huismans}, {\em {DOUAR:} {A} new three-dimensional creeping flow numerical
  model for the solution of geological problems}, Physics of the Earth and
  Planetary Interiors, 171 (2008), pp.~76--91.

\bibitem{brenner_korns_2004}
{\sc S.~C. {B}renner}, {\em {K}orn's inequalities for piecewise {${H}^1$}
  vector fields}, {M}athematics of Computation, 73 (2004), pp.~pp.~1067--1087.

\bibitem{brezzi_mixed_1991}
{\sc F.~{B}rezzi and M.~{F}ortin}, {\em {{M}ixed and hybrid finite element
  methods}}, {S}pringer {N}ew {Y}ork, {N}ew {Y}ork, {NY}, 1991.

\bibitem{brown_gmres_1997}
{\sc P.~N. Brown and H.~F. Walker}, {\em {GMRES} on (nearly) singular systems},
  SIAM Journal on Matrix Analysis and Applications, 18 (1997), pp.~37--51.

\bibitem{burs.etal.1691.2009}
{\sc C.~Burstedde, O.~Ghattas, G.~Stadler, T.~Tu, and L.~C. Wilcox}, {\em
  Parallel scalable adjoint-based adaptive solution of variable-viscosity
  {S}tokes flow problems}, Computer Methods in Applied Mechanics and
  Engineering, 198 (2009), pp.~1691--1700.

\bibitem{cockburn_note_2007}
{\sc B.~Cockburn, G.~Kanschat, and D.~Sch\"{o}tzau}, {\em A {Note} on
  {Discontinuous} {Galerkin} {Divergence}-free {Solutions} of the
  {Navier}–{Stokes} {Equations}}, Journal of Scientific Computing, 31 (2007),
  pp.~61--73.

\bibitem{dobrev_twolevel_2006}
{\sc V.~A. {D}obrev, R.~D. {L}azarov, P.~S. {V}assilevski, and L.~T.
  {Z}ikatanov}, {\em {T}wo-level preconditioning of discontinuous {{G}alerkin}
  approximations of second-order elliptic equations}, {N}umerical Linear
  Algebra with Applications, 13 (2006), pp.~753--770.

\bibitem{drosson_stability_2013}
{\sc M.~{D}rosson and K.~{H}illewaert}, {\em {O}n the stability of the
  symmetric interior penalty method for the {S}palart--{A}llmaras turbulence
  model}, {J}ournal of Computational and Applied Mathematics, 246 (2013),
  pp.~122--135.

\bibitem{duretz_discretization_2011}
{\sc T.~{D}uretz, D.~A. {M}ay, T.~V. {G}erya, and P.~J. {T}ackley}, {\em
  {D}iscretization errors and free surface stabilization in the finite
  difference and marker-in-cell method for applied geodynamics: {A} numerical
  study: {FD}-{MIC} {scheme} {discretization} {errors}}, {G}eochemistry,
  Geophysics, Geosystems, 12 (2011).

\bibitem{Duretz01032016}
{\sc T.~Duretz, D.~A. May, and P.~Yamato}, {\em A free surface capturing
  discretization for the staggered grid finite difference scheme}, Geophysical
  Journal International, 204 (2016), pp.~1518--1530.

\bibitem{elma.iffis:2005}
{\sc H.~Elman, D.~Silvester, and A.~Wathen}, {\em {F}inite elements and fast
  iterative solvers}, Oxford Univ. Press, New York, 2005.

\bibitem{elman_finite_2014}
{\sc H.~C. Elman, D.~J. Silvester, and A.~J. Wathen}, {\em Finite elements and
  fast iterative solvers: with applications in incompressible fluid dynamics},
  Oxford University Press (UK), 2014.

\bibitem{full:gji.1.1995}
{\sc P.~Fullsack}, {\em An arbitrary {L}agrangian-{E}ulerian formulation for
  creeping flows and its application in tectonic models}, Geophysical Journal
  Interational, 120 (1995), pp.~1--23.

\bibitem{gerya2013adaptive}
{\sc T.~V. Gerya, D.~A. May, and T.~Duretz}, {\em An adaptive staggered grid
  finite difference method for modeling geodynamic stokes flows with strongly
  variable viscosity}, Geochemistry, Geophysics, Geosystems, 14 (2013),
  pp.~1200--1225.

\bibitem{Gerya03}
{\sc T.~V. Gerya and D.~A. Yuen}, {\em Charaterictics-based marker method with
  conservative finite-difference schemes for modeling geological flows with
  strongly variable transport properties}, Physiscs of the Earth and Planetary
  Interiors, 140 (2003), pp.~293--318.

\bibitem{hillewaert_development_2013}
{\sc K.~{H}illewaert}, {\em {Development} {of} {the} {discontinuous} {Galerkin}
  {method} {for} {high}--{resolution}, {large} {scale} {CFD} {and} {acoustics}
  {in} {industrial} {geometries}}, PhD thesis, {U}niversit\'e {C}atholique de
  {L}ouvain, Feb. 2013.

\bibitem{kanschat_multigrid_2015}
{\sc G.~Kanschat and Y.~Mao}, {\em Multigrid methods for
  {$H^\text{div}$}-conforming discontinuous {Galerkin} methods for the {Stokes}
  equations}, Journal of Numerical Mathematics, 23 (2015), pp.~51--66.

\bibitem{karato1997phase}
{\sc S.-I. Karato}, {\em Phase transformations and rheological properties of
  mantle minerals}, Earth’s Deep Interior, 7 (1997), pp.~223--272.

\bibitem{kronbichler_high_2012}
{\sc M.~{K}ronbichler, T.~{H}eister, and W.~{B}angerth}, {\em {H}igh accuracy
  mantle convection simulation through modern numerical methods: {{H}igh}
  accuracy mantle convection simulation}, {G}eophysical Journal International,
  191 (2012), pp.~12--29.

\bibitem{GJI:GJI5164}
{\sc S.~M. Lechmann, D.~A. May, B.~J.~P. Kaus, and S.~M. Schmalholz}, {\em
  Comparing thin-sheet models with {3-D} multilayer models for continental
  collision}, Geophysical Journal Interational, 187 (2011), pp.~10--33.

\bibitem{lehmann_comparison_2015}
{\sc R.~S. Lehmann, M.~Lukáčová-Medvid'ová, B.~J.~P. Kaus, and A.~Popov},
  {\em Comparison of continuous and discontinuous {Galerkin} approaches for
  variable-viscosity {Stokes} flow: {Comparison} of {CG} and {DG} approaches
  for variable-viscosity {Stokes} flow}, ZAMM - Journal of Applied Mathematics
  and Mechanics / Zeitschrift für Angewandte Mathematik und Mechanik,  (2015).

\bibitem{leng_zhon_Q04006_2011}
{\sc W.~Leng and S.~J. Zhong}, {\em Implementation and application of adaptive
  mesh refinement for thermochemical mantle convection studies}, Geochemistry,
  Geophysics, Geosystems, 12 (2011), p.~Q04006.

\bibitem{may2015scalable}
{\sc D.~A. May, J.~Brown, and L.~Le~Pourhiet}, {\em A scalable, matrix-free
  multigrid preconditioner for finite element discretizations of heterogeneous
  stokes flow}, Computer Methods in Applied Mechanics and Engineering, 290
  (2015), pp.~496--523.

\bibitem{may2008preconditioned}
{\sc D.~A. May and L.~Moresi}, {\em Preconditioned iterative methods for
  {S}tokes flow problems arising in computational geodynamics}, Physics of the
  Earth and Planetary Interiors, 171 (2008), pp.~33--47.

\bibitem{more.dufo.ea:jcp.476.2003}
{\sc L.~Moresi, F.~Dufour, and H.-B. M{\"u}hlhaus}, {\em A {L}agrangian
  integration point finite element method for large deformation modeling of
  viscoelastic geomaterials}, Journal of Computational Physics, 184 (2003),
  pp.~476--497.

\bibitem{moresi_computational_2007}
{\sc L.~{M}oresi, S.~{Q}uenette, V.~{L}emiale, C.~M{\'e}riaux, B.~{A}ppelbe,
  and H.-B. M{\"u}hlhaus}, {\em {C}omputational approaches to studying
  non-linear dynamics of the crust and mantle}, {P}hysics of the {E}arth and
  Planetary Interiors, 163 (2007), pp.~69--82.

\bibitem{poli.podl:gji.553.1992}
{\sc A.~Poliakov and Y.~Podladchikov}, {\em Diapirism and topography},
  Geophysical Journal International, 109 (1992), pp.~553--564.

\bibitem{Popov200855}
{\sc A.~A. Popov and S.~V. Sobolev}, {\em {SLIM3D}: A tool for
  three-dimensional thermomechanical modeling of lithospheric deformation with
  elasto-visco-plastic rheology}, Physiscs of the Earth and Planetary
  Interiors, 171 (2008), pp.~55--75.
\newblock {R}ecent Advances in Computational Geodynamics: Theory, Numerics and
  Applications.

\bibitem{ranalli1995rheology}
{\sc G.~Ranalli}, {\em Rheology of the {E}arth}, Springer Science \& Business
  Media, 1995.

\bibitem{schotzau_mixed_2002}
{\sc D.~{S}ch\"{o}tzau, C.~{S}chwab, and A.~{T}oselli}, {\em {M}ixed
  $hp$-{DGFEM} for incompressible flows}, {SIAM} Journal on Numerical Analysis,
  40 (2002), pp.~2171--2194.

\bibitem{schotzau_mixed_2004}
\leavevmode\vrule height 2pt depth -1.6pt width 23pt, {\em {M}ixed $hp$-{DGFEM}
  for incompressible flows {II}: {{G}eometric} edge meshes}, IMA Journal of
  Numerical Analysis, 24 (2004), p.~273.

\bibitem{schubert_mantle_2001}
{\sc G.~{S}chubert, D.~L. {T}urcotte, and P.~{O}lson}, {\em {M}antle
  {convection} in the {{E}arth} and {planets}}, {C}ambridge University Press,
  2001.
\newblock {C}ambridge books online.

\bibitem{stem.etal.223.2006}
{\sc K.~Stemmer, H.~Harder, and U.~Hansen}, {\em A new method to simulate
  convection with strongly temperature- and pressure-dependent viscosity in a
  spherical shell: applications to the {E}arth's mantle}, Physics of the Earth
  and Planetary Interiors, 157 (2006), pp.~223--249.

\bibitem{tack.7.2008}
{\sc P.~J. Tackley}, {\em Modelling compressible mantle convection with large
  viscosity contrasts in a three-dimensional spherical shell using the yin-yang
  grid}, Physics of the Earth and Planetary Interiors, 171 (2008), pp.~7--18.

\bibitem{toselli_hp_2002}
{\sc A.~{T}oselli}, {\em {$hp$ discontinuous {G}alerkin approximations for the
  {S}tokes problem}}, {M}athematical Models and Methods in Applied Sciences, 12
  (2002), pp.~1565--1597.

\bibitem{turc.torr.431.1973}
{\sc D.~L. Turcotte, K.~E. Torrance, and A.~T. Hsui}, {\em Methods in
  {C}omputational {P}hysics: {G}eophysics}, vol.~13, Academic Press, Inc.,
  1973, ch.~Convection in the earth's mantle, pp.~431--454.

\bibitem{van_slingerland_fast_2014}
{\sc P.~van Slingerland and C.~Vuik}, {\em Fast linear solver for diffusion
  problems with applications to pressure computation in layered domains},
  Computational Geosciences, 18 (2014), pp.~343--356.

\bibitem{van_slingerland_scalable_2015}
{\sc P.~van {S}lingerland and C.~{V}uik}, {\em {S}calable two-level
  preconditioning and deflation based on a piecewise constant subspace for
  ({SIP}){DG} systems for diffusion problems}, {J}ournal of Computational and
  Applied Mathematics, 275 (2015), pp.~61--78.

\bibitem{wang_new_2007}
{\sc J.~Wang and X.~Ye}, {\em New {Finite} {Element} {Methods} in
  {Computational} {Fluid} {Dynamics} by {H}(div) {Elements}}, SIAM Journal on
  Numerical Analysis, 45 (2007), pp.~1269--1286.

\bibitem{wein.schm.425.1992}
{\sc R.~F. Weinberg and H.~Schmeling}, {\em Polydiapirs: {M}ultiwave length
  gravity structures}, Journal of Structural Geolgy, 14 (1992), pp.~425--436.

\bibitem{Zaleski199255}
{\sc S.~Zaleski and P.~Julien}, {\em Numerical simulation of
  {R}ayleigh-{T}aylor instability for single and multiple salt diapirs},
  Tectonophysics, 206 (1992), pp.~55--69.

\bibitem{zhong_analytic_1996}
{\sc S.~{Z}hong}, {\em {{A}nalytic solutions for {S}tokes' flow with lateral
  variations in viscosity}}, {G}eophysical Journal International, 124 (1996),
  pp.~18--28.

\end{thebibliography}

\end{document}